\documentclass[a4paper,12pt]{article}

\usepackage{mymacros}

\newcommand{\tr}{\operatorname{tr}}

\newcommand{\ad}{\operatorname{ad}}
\newcommand{\bsxi}{{\boldsymbol{\xi}}}
\newcommand{\bseta}{{\boldsymbol{\eta}}}

\newcommand{\Sec}{\operatorname{Sec}}
\newcommand{\pardeg}{\operatorname{par\,deg}}
\newcommand{\su}{\mathfrak{su}(2)}
\newcommand{\Hpar}{H_\mathrm{par}}
\newcommand{\scNa}{\scN_\bsalpha}

\newcommand{\scNp}{\scN_{\bsalpha_+}}
\newcommand{\scNm}{\scN_{\bsalpha_-}}
\newcommand{\scNz}{\scN_{\bsalpha_0}}
\newcommand{\scNpm}{\scN_{\bsalpha_\pm}}
\newcommand{\Sigmap}{\Sigma_{\bsalpha_+}}
\newcommand{\Sigmam}{\Sigma_{\bsalpha_-}}
\newcommand{\Sigmaz}{\Sigma_{\bsalpha_0}}
\newcommand{\Sigmapm}{\Sigma_{\bsalpha_\pm}}
\newcommand{\scMp}{\scM_{\bsw_+}}
\newcommand{\scMm}{\scM_{\bsw_-}}
\newcommand{\scMz}{\scM_{\bsw_0}}

\newcommand{\Xo}{X^{\circ}}

\newcommand{\RS}{\Sigma}
\newcommand{\alcove}{A_{\frakt}}

\title{Goldman systems and bending systems}
\author{Yuichi Nohara and Kazushi Ueda}
\date{}
\begin{document}

\maketitle

\begin{abstract}
We show that the moduli space
of parabolic bundles on the projective line
and the polygon space are isomorphic,
both as complex manifolds
and symplectic manifolds equipped with structures of completely integrable systems,
if the stability parameters are 
small.
\end{abstract}

\section{Introduction}

Let $\scN_\bsalpha$ be
the moduli space of semi-stable
parabolic bundles of rank 2
on the projective line $X$
with $n$ marked points
$z_1, \dots, z_n$,
where
$
 \bsalpha \in (0, 1/2)^n
$
is the parameter
for the parabolic weight.
The moduli space $\scN_\bsalpha$ is a smooth projective manifold
for a generic choice of $\bsalpha$.
Mehta and Seshadri
\cite{Mehta-Seshadri}
gave a construction of $\scN_\bsalpha$
using geometric invariant theory,
and showed that it is diffeomorphic to the moduli space
of unitary representations of the fundamental group
of the punctured Riemann surface
$X \setminus \{ z_1, \dots, z_n \}$.


With any pair-of-pants decomposition
of the punctured Riemann surface
$X \setminus \{ z_1, \dots, z_n \}$,
one can associates a completely integrable system
on $\scN_\bsalpha$
called the {\em Goldman system}
\cite{Goldman_IFLG}.
The Goldman system resembles
the moment map of a toric variety
\cite{Weitsman_RPMS,
Jeffrey-Weitsman_BSO,
Jeffrey-Weitsman_TS,
Jeffrey-Weitsman_TSII},
although the natural complex structure on $\scN_\bsalpha$
is not preserved by the action
of the Goldman's Hamiltonians.
Even worse,
the moduli space $\scN_\bsalpha$ as a complex manifold
usually does not admit a structure of a toric variety at all.

A pair-of-pants decomposition
of the punctured Riemann surface
$X \setminus \{ z_1, \dots, z_n \}$
is described by a trivalent graph $\Gamma$ with $n$ leaves,
in such a way that nodes correspond to pairs of pants and
edges show how they are glued together
as shown in \pref{fg:decomp}.
In this paper,
we consider the case when the genus of $X$ is zero,
so that $\Gamma$ is a tree.
The corresponding Goldman system will be denoted by
$\Theta_\Gamma : \scN_\bsalpha \to \bR^{n-3}$.

\begin{figure}
\centering
\input{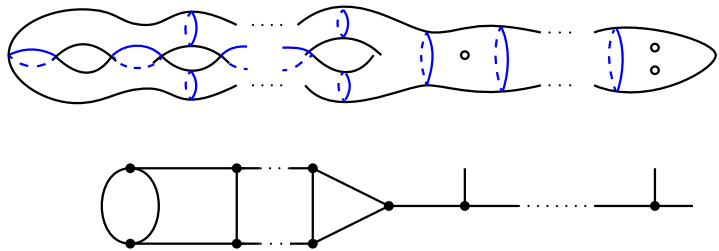}
\caption{A pair-of-pants decomposition
and its dual graph}
\label{fg:decomp}
\end{figure}

The moduli space $\scN_\bsalpha$ is closely related
to the moduli space
$\scM_\bsw$ of ordered $n$ points
on the projective line,
which is constructed
as the geometric invariant theory quotient;
$$
 \scM_\bsw
  = \Proj \lb
      \bigoplus_{k=0}^\infty
        \Gamma((\bP^1)^n, \scO(k w_1, \ldots, k w_n))^{\PGL_2}
              \rb.
$$
Here $\bsw = (w_1, \ldots, w_n) \in \bQ^n$ is the parameter
for the $\PGL_2$-linearization,
which determines the stability condition
and the ample line bundle on the quotient.

The moduli space $\scM_\bsw$ has
a natural symplectic structure
as a polarized projective variety.
As such, it admits an interpretation
as the moduli space of polygons in $\bR^3$
with side lengths $(w_1, \ldots, w_n)$.
Fix a convex planar $n$-gon $P$
called the {\em reference polygon}.
We identify the set of triangulations of the reference polygon
with the set of trivalent trees with $n$ leaves,
by assigning the dual graph to a triangulation.
For any triangulation $\Gamma$ of the reference polygon,
Klyachko \cite{Klyachko_SP} and
Kapovich and Millson \cite{Kapovich-Millson_SGPES}
introduced a completely integrable system
$$
 \Phi_\Gamma : \scM_\bsw \to \bR^{n-3}
$$
called the {\em bending system}.

To relate a completely integrable system
with a toric variety,
the notion of a {\em toric degeneration of an integrable system} is introduced
in \cite[Definition 1.1]{Nishinou-Nohara-Ueda_TDGCSPF}.
For each triangulation $\Gamma$
of the reference polygon $P$,
we have given a toric degeneration
of the corresponding bending system
in \cite[Corollary 1.3]{MR3211821}.
The toric degeneration of $\scM_\bsw$
underlying this toric degeneration of the bending system
is the one given in 
\cite{Hausmann-Knutson_PSG,
Kamiyama-Yoshida,
Speyer-Sturmfels_TG,
Foth-Hu,
Howard-Manon-Millson}.

The main result in this paper is the following:

\begin{theorem} \label{th:main}
Let
$
 \bsalpha
  = (\alpha_1, \ldots, \alpha_n)
  \in (0,1/2)^n
$
be a parabolic weight satisfying
$
 | \bsalpha |
  := \alpha_1 + \dots + \alpha_n
  < 1.
$
Then for any triangulation $\Gamma$
of the reference polygon $P$, 
there is a symplectomorphism
$\psi : \scN_\bsalpha \to \scM_\bsw$
such that $\psi^*\Phi_{\Gamma} = \Theta_{\Gamma}$.
\end{theorem}

Combining with \cite[Corollary 1.3]{MR3211821},
we have the following.

\begin{corollary}
Suppose that 
$
 \bsalpha
  \in (0,1/2)^n
$
satisfies
$
 | \bsalpha |< 1.
$
Then there exists a continuous family $\pi : \frakY \to [0,1]$ of
symplectic varieties equipped with completely integrable systems
$F_t :  Y_t = \pi^{-1}(t) \to \bR^{n-3}$ such that 
$(Y_1, F_1) = (\scN_{\bsalpha}, \Theta_{\Gamma})$,
and $(Y_0, F_0)$ is a pair of a toric variety and a toric moment map
whose moment polytope is 
$\Theta_{\Gamma}(\scN_{\bsalpha})$.
Moreover, there is a continuous family of maps $\psi_t : Y_1 \to Y_{1-t}$
which are symplectomorphisms on an open dense subset and satisfy
$\psi^*_t F_{1-t} = F_1 = \Theta_{\Gamma}$.
\end{corollary}

As a corollary,
we obtain a new proof
of the $|\bsalpha|<1$ case
of the result of Jeffrey and Weitsman
\cite{Jeffrey-Weitsman_BSO}
that the numbers of lattice points
on the moment polytope of the Goldman system
is equal to the number of sections
of the natural ample line bundle
on $\scN_\bsalpha$
provided by GIT construction.


This paper is organized as follows:
In Section \ref{sc:weighted_projective_lines},
we recall the description of coherent sheaves
on smooth rational orbifold curves
due to Geigle and Lenzing
\cite{Geigle-Lenzing_WPC},
who call such curves \emph{weighted projective lines}.
In Section \ref{sc:quasi-parabolic_bundles},
we recall the relation between quasi-parabolic bundles
and orbifold bundles.
%
In Section \ref{sc:parabolic_weights},
we recall the definition of parabolic weights and
stability conditions.
In \pref{sc:Mehta-Seshadri},
we recall the relation
between flat $SU(2)$-bundles and
parabolic bundles of rank two and parabolic degree zero.
In Section \ref{sc:distinguished_parameter},
we show that the moduli space $\scN_\bsalpha$ is
the projective space $\bP^{n-3}$
for a suitable choice of a stability parameter.
In Section \ref{sc:wall-crossing_parabolic1},
we recall wall-crossing phenomena
in moduli spaces of parabolic bundles
following Bauer \cite{Bauer_PB};
the space of stability parameters is divided
into finitely many chambers by walls,
and the change in moduli spaces
under wall-crossing can be described
explicitly as a blow-down followed by a blow-up.
More general results
on variation of geometric invariant theory quotients
are obtained by Thaddeus \cite{Thaddeus_flip} and
 Dolgachev and Hu \cite{Dolgachev-Hu}.
In Section \ref{sc:wall-crossing_parabolic2},
we use the wall-crossing phenomena
to give an explicit description of $\scN_\bsalpha$
for general $\bsalpha$.
The strategy is to start with the stability parameter
in Section \ref{sc:distinguished_parameter} and
successively cross walls in the space of stability parameters
to arrive at any stability parameter.
This is strategy is already used in Bauer \cite{Bauer_PB},
and the main difference between his work and ours is that
we make an extensive use
of the language of weighted projective lines,
developed by Geigle and Lenzing
\cite{Geigle-Lenzing_WPC},
and the chamber that we start with is different
from that of Bauer.
In Section \ref{sc:comparison},
we give a description of the moduli space $\scM_\bsw$
parallel to that of $\scN_\bsalpha$.
This immediately shows that
$\scM_\bsw$ and $\scN_\bsalpha$ are isomorphic
if $\bsw = \bsalpha$ and $|\bsalpha|<1$.
In Section \ref{sc:bending}, we recall the construction of the bending system 
on $\scM_{\bsw}$.
In Section \ref{sc:Goldman_system},
we recall the description of the symplectic structure given by
Guruprasad, Huebschmann, Jeffrey, and Weinstein
\cite{Guruprasad-Huebschmann-Jeffrey-Weinstein}, and the Goldman system.
In Section \ref{sc:extended_moduli} we recall extended moduli spaces
defined by Jeffrey \cite{Jeffrey_EMS} and
Hurtubise and Jeffrey \cite{Hurtubise-Jeffrey_RWFFPB}
to construct $\scN_{\bsalpha}$ as a finite dimensional 
symplectic reduction, and as a quasi-Hamiltonian reduction
\cite{Alekseev-Malkin-Meinrenken_LGVMM}.
In Section \ref{sc:wall_extended_moduli}, we see the walls 
in Section \ref{sc:wall-crossing_parabolic1} from the view point of
quasi-Hamiltonian reduction.
In Section \ref{sc:gluing}, we study the Goldman system via gluing 
of Riemann surface, following the idea of \cite{Hurtubise-Jeffrey_RWFFPB}
and \cite{Alekseev-Malkin-Meinrenken_LGVMM}.
In Section \ref{sc:isom_Goldman}, we construct a symplectomorphism
between $\scN_{\bsalpha}$ and $\scN_{t \bsalpha}$ $(0<t \le 1)$
which identifies the Goldman systems  
in the case where $|\bsalpha|<1$.
In Section \ref{sc:Goldman-bending}, we show that $\scM_\bsalpha$ and $\scN_\bsalpha$
are symplectomorphic in such a way that 
the Goldman system on $\scN_\bsalpha$
and the bending system on $\scM_\bsw$ are identified
for sufficiently small $\bsalpha$.
Combining with the result in Section \ref{sc:isom_Goldman}, 
Theorem \ref{th:main} is proved.

\emph{Acknowledgment}:
We thank Shinnosuke Okawa
for valuable discussions,
and the anonymous referee for suggesting improvements.
A part of this paper is written
while K.~U. is visiting Korea Institute
for Advanced Study,
whose hospitality and financial support
is gratefully acknowledged.
This research is supported by
Grant-in-Aid for Young Scientists
(No.23740055 and No.24740043).


 
\section{Orbifold projective lines}
 \label{sc:weighted_projective_lines}

Let $\bX$ be a smooth Deligne-Mumford stack of dimension one
without generic stabilizer.
We assume that $\bX$ is rational,
so that the coarse moduli space $X$ of $\bX$ is isomorphic to $\bP^1$.
Such a stack is studied in detail by Geigle and Lenzing \cite{Geigle-Lenzing_WPC}
under the name {\em weighted projective lines}, and
we summarize some of their results in this subsection.
One can also see e.g. \cite{MR2931898} and references therein
for more on this subject.
Orbifold points of $\bX$ will be denoted by $w_1, \dots, w_n$,
and their images in $X$ will be denoted by $z_1, \dots, z_n$.
The absence of generic stabilizer implies that
the stabilizer group $\Gamma_{p_i}$
at $w_i$ for any $i = 1, \dots, n$ is a cyclic group,
whose order will be denoted by $p_i$.

Locally around the orbifold point $w_i$,
we can take an orbifold chart
$
 [\bA / \Gamma_{w_i}] \hookrightarrow \bX
$
where $\bA = \Spec \bC[u]$ is an affine space and
$\Gamma_{w_i}$ acts linearly by a primitive $p_i$-th root of unity.
Following \cite{Geigle-Lenzing_WPC},
we let $\scO_\bX(\vecx_i)$ for $i = 1, \dots, n$,
denote the dual of $\scO(-\vecx_i)$,
defined as the kernel of the natural morphism
$
 \scO_\bX \to \scO_{w_i}
$
to the skyscraper sheaf
$
 \scO_{w_i}
  = \big[ \big( \Spec \bC[u]/(u) \big) \big/ \Gamma_{w_i} \big];
$
$$
 0 \to \scO_\bX(-\vecx_i) \to \scO_\bX \to \scO_{w_i} \to 0.
$$
We also define $\scO_\bX(\vecc)$ as the line bundle $\scO_\bX(x)$,
which does not depend
on the choice of a point $x \in \bX \setminus \{ w_1, \dots, w_n \}$.
One has relations
$$
 \scO_\bX(p_i \vecx_i) = \scO_\bX(\vecc), \quad i = 1, \dots, n,
$$
and the Picard group of $\bX$ is given by
$$
 L = \Pic \bX
  = \bZ \vecx_1 \oplus \dots \oplus \bZ \vecx_n \oplus \bZ \vecc
      / (p_1 \vecx_1 - \vecc, \dots, p_n \vecx_n - \vecc).
$$
Choose a global coordinate on $X \cong \bP^1$ so that
the points $z_1, \dots, z_n$ on $X$ are given in this coordinate
by
$\lambda_1 = \infty$,
$\lambda_2 = 0$ and
$\lambda_3, \dots, \lambda_n \in \bA^1 \setminus \{ 0 \}$.
The total coordinate ring of $\bX$ is given by
\begin{align*}
 S
  = \bigoplus_{\veck \in L} H^0(\scO_\bX(\veck))
  = k[X_0, X_1, \dots, X_n] \, \big/
       \lb X_i^{p_i} - X_2^{p_2} + \lambda_i X_1^{p_1} \rb_{i=2}^n,
\end{align*}
which is graded by the abelian group $L$ as
$
 \deg X_i = \vecx_i
$
for
$
 i = 1, \dots, n.
$
The stack $\bX$ is recovered as the quotient stack
$$
 \bX = \big[ (\Spec S \setminus \{ \bszero \}) \big/ G \big]
$$
by the affine algebraic group
$
 G = \Spec \bC[L].
$
The graded ring $S$ is Gorenstein with parameter
$
 \vecomega = (n-2)-\sum_{i=1}^n \vecx_i,
$
and Serre duality on $\bX$ is given by
$$
 \Ext^1(\scE, \scF)
  \cong H^0(\scF, \scE \otimes \scO_\bX(\vecomega))^\vee
$$
for any coherent sheaves $\scE$ and $\scF$.

\section{Quasi-parabolic bundles as orbifold bundles}
 \label{sc:quasi-parabolic_bundles}

In this section,
we recall the relation between quasi-parabolic bundles on punctured curves
and orbifold bundles on orbi-curves.
Although this is well-known to experts
and essentially goes back to \cite{Mehta-Seshadri},
we provide a sketch of proof here for the readers' convenience.

Let $\Utilde = \Spec \bC[u]$ be an affine line and
$\bU = [\Utilde / \Gamma]$ be the quotient stack of $\Utilde$
with respect to the $\Gamma = \bZ / p_j \bZ$-action,
which acts on points of $\Utilde$ as
\begin{align} \label{eq:Gamma_action}
 u \mapsto \zeta^{-1} u, \quad
 \zeta = \exp(2 \pi \sqrt{-1} / p_j).
\end{align}
A complex analytic neighborhood of the origin in $\bU$ is identified
with a complex analytic neighborhood of $w_j$ in $\bX$.
The coarse moduli space of $\bU$ is given by
$
 U = \Spec \bC[v],
$
where $\bC[v] = \bC[u]^\Gamma$
for $v = u^{p_j}$ is the invariant ring.

The action of $\Gamma$ on $\Utilde$
induces an action on the coordinate ring $\bC[u]$
in such a way that an element $\gamma \in \Gamma$
sends a function $f$
to its pull-back $(\gamma^{-1})^* f$ by $\gamma^{-1} \colon \Utilde \to \Utilde$.
It follows from the definition of sheaves on quotient stacks
that a locally-free sheaf $\scE$ on $\bU$ corresponds to
a $\Gamma$-equivariant locally-free sheaf on $\Utilde$.
Since $\Utilde$ is affine,
a $\Gamma$-equivariant locally-free sheaf on $\Utilde = \Spec \bC[u]$
is the same as a free $\bC[u]$-module $M$,
equipped with an action of $\Gamma$
satisfying
\begin{align} \label{eq:equiv1}
 \gamma \cdot (f m) = (\gamma \cdot f) (\gamma \cdot m),
\end{align}
for any $\gamma \in \Gamma$, $f \in \bC[u]$ and $m \in M$.
Here $\cdot$ is the $\Gamma$-action on $\bC[u]$ and $M$.
The crossed product algebra $\bC[u] \rtimes \Gamma$
consists of elements of the form $f \otimes \gamma$
for $f \in \bC[u]$ and $\gamma \in \Gamma$,
with relations
\begin{align} \label{eq:equiv2}
 (f \otimes \gamma) \circ (g \otimes \delta)
  = f (\gamma \cdot g) \otimes (\gamma \delta).
\end{align}
It follows from \eqref{eq:equiv1} and \eqref{eq:equiv2}
that a $\Gamma$-equivariant $\bC[u]$-module can be identified
with a $\bC[u] \rtimes \Gamma$-module.

Let $P$ be a finitely-generated $\bC[u] \rtimes \Gamma$-module.
As a $\Gamma$-module,
it has a direct sum decomposition
$
 P = \bigoplus_{i=1}^{p_j} P_i 
$
into isotypical components,
where the generator $[1] \in \Gamma$ acts on $P_i$
by multiplication by
$\exp(2 \pi \sqrt{-1} (i-1) / p_j)$.
The $\bC[u]$-module structure is determined
by the action of $u$,
which is just a collection of $\bC$-linear maps
$$
 u : P_i \to P_{i-1}, \qquad
 i \in \bZ / p_j \bZ.
$$
Each $P_i$ is a $\bC[v]$-module,
and multiplication by $u$ is a homomorphism
of $\bC[v]$-modules,
which must satisfy
$$
 u^m = v : P_i \to P_{i-m}.
$$
In terms of sheaves $\scP_i$
of $\scO_U$-modules
associated with $\bC[v]$-modules $P_i$,
this gives a {\em quasi-parabolic sheaf},
defined as an infinite sequence
\begin{equation} \label{eq:parabolic_str}
 \cdots
  \xto{u} \scP_i
  \xto{u} \scP_{i+1}
  \xto{u} \cdots
\end{equation}
such that $\scP_{i+p_j} = \scP_i(-z_j)$ and
the composition
$$
 \scP_{i+p_j} \xto{u^{p_j}} \scP_i
$$
is equal to the multiplication
$$
 \scP_i(-z_j) \xto{v} \scP_i
$$
by $v$ for any $i \in \bZ$.
A \emph{morphism} of quasi-parabolic sheaves
is a collection
$f_i \colon \cP_i \to \cQ_i$ of morphisms
making the diagram
\begin{align*}
\begin{CD}
 \cdots @>{u}>> \cP_i @>{u}>> \cP_{i+1} @>{u}>> \cdots \\
  @. @V{f_i}VV @V{f_{i+1}}VV \\
 \cdots @>{u}>> \cQ_i @>{u}>> \cQ_{i+1} @>{u}>> \cdots \\
\end{CD}
\end{align*}
commutative.
Under the correspondence
between $\bC[v]$-modules with quasi-parabolic structures
and $\bC[u] \rtimes \Gamma$-modules,
a morphism of quasi-parabolic sheaves
can be identified with a morphism of $\bC[u] \rtimes \Gamma$-modules.
By using this correspondence around each orbifold points,
one obtains the following:

\begin{proposition} \label{pr:q-par}
The category of quasi-parabolic sheaves on $X$
is equivalent to the category of coherent sheaves on $\bX$.
\end{proposition}

If $P$ is locally-free,
then multiplication by $v$ is an injection,
so that \eqref{eq:parabolic_str} gives a filtration
$$
 \scP_1(-z_j) \cong \scP_{p_j+1}
  \hookrightarrow \scP_{p_j}
  \hookrightarrow \cdots
  \hookrightarrow \scP_2
  \hookrightarrow \scP_1
$$
of sheaves,
which in turn gives a filtration
$$
 0
 = F_{{p_j}+1}(\scP_{z_j})
  \subset F_{p_j}(\scP_{z_j})
  \subset \cdots
  \subset F_1(\scP_{z_j})
  = \scP_{z_j}
$$
of the fiber $\scP_{z_j} = \scP_1 / v \cdot \scP_1$
of $\scP_1$ at $z_j$.
A pair consisting of a locally-free sheaf and
a filtration at each $z_j$ is called
a {\em quasi-parabolic bundle}.
A morphism of quasi-parabolic bundles $\scP$ and $\scQ$
can be described, in terms of a filtration at each $z_j$,
as a morphism $\varphi$ of the underlying vector bundle
such that $\varphi(F_i(\scP_{z_j})) \subset F_i(\scQ_{z_j})$.
The equivalence in Proposition \ref{pr:q-par} restricts
to an equivalence
between the category of vector bundles on $\bX$ and
the category of quasi-parabolic bundles on $X$.

\section{Parabolic weights and stability conditions}
 \label{sc:parabolic_weights}

Assume that the stabilizer groups
at all orbifold points are cyclic groups
of order two;
$\Gamma_{w_j} = \bZ / 2 \bZ$
for $j = 1, \dots, n$.
A vector bundle on $\bX$ corresponds
to a quasi-parabolic bundle
consisting vector bundle $\scP$ on $X$
and 2-step flags
$$
 0 = F_3(\scP_{z_j})
 \subset F_2(\scP_{z_j})
 \subset F_1(\scP_{z_j}) = \scP_{z_j}
$$
for each $j=1,\ldots,n$.
The Picard group of $\bX$ is given by
$$
 L = \Pic \bX
  = \bZ \vecx_1 \oplus \dots \oplus \bZ \vecx_n \oplus \bZ \vecc
      / (2 \vecx_1 - \vecc, \dots, 2 \vecx_n - \vecc).
$$
The structure sheaf $\scO_\bX$ corresponds
to the trivial bundle $\scP = \scO_X$
equipped with the filtration
$F_2(\scP_{z_j}) = 0$
for any $z_j$.
The line bundle $\scO_\bX(\vecx_i)$ corresponds
to the trivial bundle $\scP = \scO_X$
equipped with the filtration
\begin{align*}
 F_2(\scP_{z_j}) = 
  \begin{cases}
   \scP_{z_j} & i = j, \\
   0 & \text{otherwise}.
  \end{cases}
\end{align*}
A parabolic bundle is a quasi-parabolic bundle
together with a choice of parabolic weights
$$
 (a_{j, 1}, a_{j, 2}) \in \bQ^2, \qquad
 0 \le a_{j, 1} < a_{j,2} < 1
$$
for each $j=1,\ldots,n$.
We always assume that
a parabolic weight satisfies
$a_{j,1} + a_{j,2} = 1$
for $j = 1, \dots, n$
in this paper.
Any subbundle $\scE$ of a parabolic bundle $\scP$ has
a natural parabolic structure
whose quasi-parabolic structure is given by
$$
 F_i(\scE_{z_j}) = F_i(\scP_{z_j}) \cap \scE_{z_j}
$$
with the same parabolic weight as $\scP$.
The {\em parabolic degree} of $\scP$ is defined by
$$
 \pardeg \scP = \deg \scP + \sum_{j=1}^n \left[
  a_{j,1} (\dim F_1(\scP_{z_j}) - \dim F_2(\scP_{z_j}))
   + a_{j,2} \dim F_2(\scP_{z_j}) \right].
$$
For example,
if $\rank \scP = 2$ and
$$
 \dim F_1(\scP_{z_j}) - \dim F_2(\scP_{z_j})
  = \dim F_2(\scP_{z_j})
  = 1, \qquad
 j = 1, \ldots, n,
$$
then the parabolic degree of $\scP$ is given by
$$
 \pardeg \scP
 = \deg \scP + \sum_{j=1}^n (a_{j,1} + a_{j,2})
 = \deg \scP + n.
$$
A parabolic bundle is {\em semi-stable}
if one has
\begin{align} \label{eq:slope}
 \frac{\pardeg \scE}{\rank \scE}
  \le \frac{\pardeg \scP}{\rank \scP}
\end{align}
for any subbundle $\scE \subset \scP$.
It is {\em stable}
if the strict inequality holds in \eqref{eq:slope}
for any non-trivial subbundle $0 \ne \scE \subsetneq \scP$.

The Picard group $L$ of $\bX$ acts on $\bQ^n$ by
$$
 \vecx_i(\bsalpha) = \bsalpha', \quad
 \alpha_j' = 
  \begin{cases}
   \alpha_j & i \ne j, \\
   - \alpha_j & i = j.
  \end{cases}
$$
Note that this action factors through
$L / (2 \vecc) \cong (\bZ / 2 \bZ)^n$.
Any element of $L$ can be written uniquely as
$
 \veck
  = k_1 \vecx_1 + \dots + k_n \vecx_n + k_0 \vecc,
$
where $k_i \in \{ 0, 1 \}$ for $i = 1, \dots, n$ and $k_0 \in \bZ$,
and the parabolic degree of the line bundle $\scO(\veck)$
is given by
$$
 \pardeg_\bsalpha \scO(\veck)
  = \left| \veck \right| + \left| \veck(\bsalpha) \right|
  := k_0 + k_1 + \dots + k_n 
       + (-1)^{k_1} \alpha_1 + \dots + (-1)^{k_n} \alpha_n.
$$

\section{Moduli spaces of parabolic bundles}
 \label{sc:Mehta-Seshadri}

Mehta and Seshadri
\cite{Mehta-Seshadri}
constructed the moduli space $\scN_\bsalpha$
of semistable parabolic bundles,
which is a normal projective variety 
parametrizing S-equivalence classes
of semistable parabolic bundles.
They have also shown that the open subvariety
$\scN_\bsalpha^s \subset \scN_\bsalpha$
parametrizing stable parabolic bundles
of parabolic degree zero is diffeomorphic
to the moduli space
of irreducible unitary representations
of the fundamental group of $\Xo := X \setminus \{ z_1, \dots, z_n \}$;
\begin{align} \label{eq:Mehta-Seshadri}
 \scN_\bsalpha^s \cong
  \lc \rho \in \Hom(\pi_1(\Xo), SU(2))_{\mathrm{irred}} \relmid
   \rho(\gamma_j) \in \cC_{\alpha_j} \rc /_\sim .
\end{align}
Here $\gamma_j \in \pi_1(\Xo)$ is a loop around $z_j$, and
$\cC_{\alpha_j} \subset SU(2)$ is the conjugacy class
containing $\exp \ld 2 \pi \sqrt{-1} \diag(a_{j,1}, a_{j,2}) \rd$.
The equivalence relation $\sim$ is defined by conjugation;
two representations $\rho$ and $\rho'$ are equivalent
if there is some $g \in SU(2)$ such that $\rho'(\gamma) = g \rho(\gamma) g^{-1}$
for any $\gamma \in \pi_1(\Xo)$.
A parabolic weight is {\em generic}
if semistability implies stability.
If the parabolic weight $\bsalpha$ is generic,
then the moduli space $\scN_\bsalpha$ is smooth.

The diffeomorphism \eqref{eq:Mehta-Seshadri} is given as follows:
For any irreducible unitary representation $\rho$ of $\pi_1(\Xo)$,
one has the flat $\bC^2$-bundle $E_\rho$ on $\Xo$ associated with $\rho$.
By tensoring $E_\rho$ with the structure sheaf $\scO_{\Xo}$
over the constant sheaf $\bC_{\Xo}$,
one obtains a coherent sheaf
$\scE^\circ := E_\rho \otimes_{\bC_{\Xo}} \scO_{\Xo}$
on $\Xo$.
Around each puncture $z_j \in X$,
we take a coordinate $v$ centered at $z_j$,
and consider the universal cover
$$
 \lc x + \sqrt{-1} y \in \bC \relmid y \gg 1 \rc \to \Xo, \quad
 x + \sqrt{-1} y \mapsto v = \exp \ld 2 \pi \sqrt{-1} \lb x+\sqrt{-1}y \rb \rd
$$
of a small disk centered at $z_j$.
Let $g = \rho(\gamma_j) \in SU(2)$ be the holonomy
of the flat bundle $E_\rho$ around $z_j$.
Then a holomorphic section of $\scE^{\circ}$ near $z_j$
is a holomorphic function
$
 f : \lc x + \sqrt{-1} y \in \bC \relmid y \gg 1 \rc \to \bC^2
$
satisfying
$
 f \lb \lb x+1 \rb + \sqrt{-1} y \rb
  = g \cdot f \lb x + \sqrt{-1} y \rb,
$
and one defines the locally-free extension $\scE$ of $\scE^{\circ}$
by saying that $f$ gives a holomorphic section of $\scE$ near $z_j$
if $f$ is bounded.
By a suitable choice of a coordinate of $\bC^2$,
one can assume that $g$ is diagonal;
$
 g = \exp \ld 2 \pi \sqrt{-1} \diag(a_{j,1}, a_{j,2}) \rd.
$
Then the space of holomorphic sections of $\scE$ is spanned by
$v \mapsto (v^{\alpha_j+k}, v^{(1-\alpha_j)+l})$
for non-negative integers $k$ and $l$.
The quasi-parabolic structure of $\scE$ at $z_j$ is defined
as the one-dimensional subspace $\bC \cdot (1, 0)$
in
$$
 \left.
 \bigoplus_{k, l=0}^\infty \bC \cdot \lb v^{\alpha_j+k}, v^{(1-\alpha_j)+l} \rb
  \right/
 v \cdot \bigoplus_{k, l=0}^\infty \bC \cdot \lb v^{\alpha_j+k}, v^{(1-\alpha_j)+l} \rb
  \cong \bC^2.
$$

\section{The moduli space for a distinguished stability parameter}
 \label{sc:distinguished_parameter}

Let $\scN_\bsalpha$ be the moduli space
of semistable parabolic bundles
of rank two and parabolic degree zero
on $X=\bP^1$ with $n$ marked points $(z_1, \dots, z_n)$.
Here, the stability parameter
$
 \bsalpha = (\alpha_1, \ldots, \alpha_n) \in (0, 1/2)^n
$
is related to the parabolic weight
$
 a = ((a_{i,1}, a_{i,2}), \ldots, (a_{n,1}, a_{n,2})) \in ((0, 1) \times (0, 1))^n
$
by
$
 (a_{i,1}, a_{i,2}) = (\alpha_i, 1 - \alpha_i).
$
The vector bundle $\scP$ on $\bX$
corresponding to a parabolic bundle in $\scN_\bsalpha$
has the same class as $\scO \oplus \scO(- \vecs)$
in the Grothendieck group $K(\bX)$
where $\vecs = \vecx_1 + \dots + \vecx_n$.
Consider the line bundle $\scL = \scO(-\vecs+\vecx_n)$.
Since
\begin{align*}
 H^0(\scO(-\vecx_n))
  &= 0, \\
 H^1(\scO(-\vecx_n))
  &\cong H^0(\scO(\vecomega+\vecx_n))^\vee \\
  &= H^0(\scO((n-2)\vecc-\vecs+\vecx_n)^\vee \\
  &= 0, \\
 H^0(\scO(\vecs - \vecx_n))
  &\cong \bC, \\
 H^1(\scO(\vecs - \vecx_n ))
  &\cong H^0(\scO(\vecomega - \vecs + \vecx_n))^\vee \\
  &= H^0(\scO((n-2) \vecc - \vecs - \vecs + \vecx_n))^\vee \\
  &= H^0(\scO((n-2) \vecc - n \vecc + \vecx_n))^\vee \\
  &= H^0(\scO(\vecx_n - 2 \vecc))^\vee \\
  &= 0,
\end{align*}
where
$
 \scO(\vecomega)
  = \scO((n-2)\vecc-\vecs)
$
is the dualizing sheaf,
one has
\begin{align*}
 \chi(\scL, \scP)
  =\chi(\scL, \scO \oplus \scO(\vecs-n))
  = \chi(\scO(\vecs - \vecx_n)) + \chi(\scO(-\vecx_n))
  = 1,
\end{align*}
so that
$$
 \Hom(\scL, \scP) \ne 0.
$$
By taking the saturation of the image of a non-zero morphism
$\phi \in \Hom(\scL, \scP)$,
one obtains a subbundle of $\scP$ of the form $\scL(\veck)$
where $\veck \in \bN \vecx_1 + \dots + \bN \vecx_n$.
Note that
$$
 \pardeg_\bsalpha \scL(\veck)
  > \pardeg_\bsalpha \scL
  = - \alpha_1 - \dots - \alpha_{n-1} + \alpha_n,
$$
so that $\scL(\veck)$ destabilizes $\scP$
if
\begin{equation*} 
 \alpha_1 + \dots + \alpha_{n-1} < \alpha_n.
\end{equation*}
This defines a chamber
in the space of stability parameters,
where every bundle is unstable and
$\scN_\bsalpha = \emptyset$.
The quotient bundle is
$$
 \scQ
  = \scP / \scL(\veck)
  \cong \scO(- \vecx_n-\veck),
$$
and the destabilizing sequence is
$$
 0
  \to \scO(- \vecs + \vecx_n + \veck)
  \to \scP
  \to \scO(- \vecx_n - \veck)
  \to 0.
$$
Consider vector bundles
obtained as extensions
$$
 0 \to \scO(-\vecs+\vecx_n) \to \scP \to \scO(-\vecx_n) \to 0,
$$
which are classified by
\begin{align*}
 e_\scP
  &\in
 \Ext^1(\scO(-\vecx_n), \scO(-\vecs+\vecx_n)) \\
  &=
 H^1(\scO(-\vecs+\vecx_n+\vecx_n)) \\
  &=
 H^1(\scO(\vecc-\vecs)) \\
  &=
 H^0(\scO((n-2)\vecc-\vecs-\vecc+\vecs))^\vee \\
  &=
 H^0(\scO((n-3)\vecc))^\vee .
\end{align*}
Given a morphism
$$
\begin{CD}
 0 @>>> \scO(-\vecs+\vecx_n) @>>> \scP @>>> \scO(-\vecx_n) @>>> 0 \\
 @. @. @VVV @. @. \\
 0 @>>> \scO(-\vecs+\vecx_n) @>>> \scP' @>>> \scO(-\vecx_n) @>>> 0 \\
\end{CD}
$$
between two such bundles $\scP$ and $\scP'$,
one obtains a diagram
$$
\begin{CD}
 0 @>>> \scO(-\vecs+\vecx_n) @>>> \scP @>>> \scO(-\vecx_n) @>>> 0 \\
 @. @VVV @VVV @VVV @. \\
 0 @>>> \scO(-\vecs+\vecx_n) @>>> \scP' @>>> \scO(-\vecx_n) @>>> 0 \\
\end{CD}
$$
since
\begin{align*}
 \Hom(\scO(-\vecs+\vecx_n), \scO(-\vecx_n))
  = H^0(\scO(\vecs - 2 \vecx_n) = 0.
\end{align*}
It follows that the isomorphism classes of such $\scP$
are classified by
$$
 \bP \Ext^1(\scO(-\vecx_n), \scO(-\vecs+\vecx_n))
  \cong \bP H^0(\scO((n - 3)\vecc)^\vee
  \cong \Sym^{n-3} \bP^1
  \cong \bP^{n-3}.
$$
\begin{proposition}
 \label{pr:distinguished_parameter}
One has
$
 \scN_\bsalpha \cong \bP^{n-3}
$
if
$
 2 \alpha_n < \left| \bsalpha \right| < 1
$
and
$
 \left| \bsalpha \right| - 2 \alpha_i - 2 \alpha_n < 0
$
for any $i = 1, \dots, n-1$.
\end{proposition}

\begin{proof}
Let $\scP$ be a rank 2 bundle on $\bX$
obtained as an extension
\begin{equation} \label{eq:P}
 0 \to \scO(-\vecs+\vecx_n) \to \scP \to \scO(-\vecx_n) \to 0.
\end{equation}
Note that
$$
 \pardeg_\bsalpha \scO(-\vecs+\vecx_n)
  = - \alpha_1 - \dots - \alpha_{n-1} + \alpha_n
  = - \left| \bsalpha \right| + 2 \alpha_n,
$$
so that $\scO(-\vecs+\vecx_n)$ does not destabilize $\scP$ if
$
 2 \alpha_n < \left| \bsalpha \right|.
$
If a line bundle $\scL$ other than $\scO(-\vecs+\vecx_n)$ has
a non-trivial morphism to $\scP$,
then $\scL$ has a non-trivial morphism to $\scO(-\vecx_n)$,
so that it can be written as $\scO(-\vecx_n-\veck)$
for some
$
 \veck
  = k_1 \vecx_1 + \dots + k_n \vecx_n + k_0 \vecc
$
where $k_i \in \{ 0, 1 \}$ for $i = 1, \dots, n$ and
$k_0 \in \bN$.
Its parabolic degree is given by
$$
 \pardeg_\bsalpha \scO(-\vecx_n - \veck)
  =
\begin{cases}
 - k_0 + \left| \bsalpha \right| - 2 \sum_{i \in I} \alpha_i - 2 \alpha_n
  & k_n = 0, \\
 - k_0 - 1 + \left| \bsalpha \right| - 2 \sum_{i \in I} \alpha_i
  & k_n = 1, \\
\end{cases}
$$
where
$
 I = \{ i \in \{ 1, \dots, n-1 \} \mid k_i = 1 \}.
$
Note that if the extension \eqref{eq:P} does not split,
then one has $\veck \ne 0$.
For $\veck \ne 0$, the conditions
$
 \left| \bsalpha \right| - 2 \alpha_i - 2 \alpha_n < 0
$
for any $i \in \{ 1, \dots, n-1 \}$
and
$
 |\bsalpha| < 1
$
imply that
$$
 \pardeg_\bsalpha \scO(- \vecx_n - \veck)
  < 0,
$$
so that the line bundle $\scO(- \vecx_n - \veck)$
does not destabilize $\scP$.
The same condition also implies that
the line bundle $\scO(-\vecs+\vecx_n+\veck)$ destabilizes
any vector bundle $\scP$
obtained as an extension
$$
 0 \to \scO(-\vecs+\vecx_n+\veck) \to \scP
  \to \scO(-\vecx_n-\veck) \to 0
$$
for any non-zero $\veck \in \bN \vecx_1 + \dots + \bN \vecx_n$,
and Proposition \ref{pr:distinguished_parameter} is proved.
\end{proof}

\section{Wall-crossings in moduli spaces of parabolic bundles}
 \label{sc:wall-crossing_parabolic1}


The space
$
 A = [0,1/2)^n
$
of stability parameters is divided into chambers
by walls
$$
 H_{I, k} = \lc \bsalpha \in A \, \big| \,
  \textstyle{\sum_{j \in J}} \alpha_j
   - \textstyle{\sum_{i \in I}} \alpha_i = k \rc,
$$
where $I \subset \{ 1, \dots, n \}$,
$J = \{ 1, \dots, n \} \setminus I$ and $k$ is a non-negative integer.
Let $C_+$ and $C_-$ be two chambers
separated by the wall $W_{I,k}$ and
take stability parameters
$\bsalpha_+ \in C_+$,
$\bsalpha_- \in C_-$ and
$\bsalpha_0 \in W_{I,k}$.
There is a diagram
\begin{equation} \label{eq:VGIT_parabolic}
\begin{array}{c}
\begin{psmatrix}[colsep=1cm, rowsep=1cm]
 \scNp & & \scNm \\
 & \scNz
\end{psmatrix}
\psset{shortput=nab,labelsep=1pt,nodesep=3pt,arrows=->}
\ncline{1,1}{2,2}_{\phi_+}
\ncline{1,3}{2,2}^{\phi_-}
\end{array}
\end{equation}
where $\phi_\pm : \scNpm \to \scNz$ are
natural projective morphisms
sending a $\bsalpha_\pm$-stable bundle
to the S-equivalence class
of the same bundle
considered as an $\bsalpha_0$-semistable bundle.
Let $\Sigmapm \subset \scNpm$ be the subscheme
parametrizing $\bsalpha_{\mp}$-unstable bundles.

\begin{proposition}[{Bauer \cite[Proposition 2.7]{Bauer_PB}}] \label{pr:VGIT}
The following hold:
\begin{enumerate}
 \item
If we set
$
 \Sigmaz
 := \phi_+(\Sigmap),
$
then one has
$
 \Sigmaz
  = \phi_-(\Sigmam).
$
 \item
Any point in $\Sigmaz$ can be written as
$[\scS \oplus \scQ]$
where
$
 \pardeg_{\bsalpha_+}(\scS)
  = - \pardeg_{\bsalpha_+}(\scQ)
  < 0
$
and
$
 \pardeg_{\bsalpha_-}(\scS)
  = - \pardeg_{\bsalpha_-}(\scQ)
  > 0.
$
 \item
$
 \phi_+^{-1}([\scS \oplus \scQ])
  \cong \bP \Ext^1(\scQ, \scS)^\vee.
$
 \item
$
 \phi_-^{-1}([\scS \oplus \scQ])
  \cong \bP \Ext^1(\scS, \scQ)^\vee.
$
\end{enumerate}
\end{proposition}

\begin{proof}
For any bundle $\scP$ in $\Sigmap$,
let
\begin{equation} \label{eq:extension}
 0 \to \scS \to \scP \to \scQ \to 0
\end{equation}
be the $\bsalpha_-$-destabilizing sequence.
Since $\scP$ is of rank two,
both the destabilizing subbundle $\scS$
and the quotient bundle $\scQ$
are line bundles.
Any point in the fiber of $\phi_+$
above the point $[\scS \oplus \scQ] \in \scNz$
is given by the extension of the form \eqref{eq:extension}, and
any such extension is $\bsalpha_+$-stable,
so that one has
$
 \phi_+^{-1}([\scS \oplus \scQ])
  \cong \bP \Ext^1(\scQ, \scS)^\vee.
$
The fiber of $\phi_-$ is obtained
by exchanging the roles of $\scS$ and $\scQ$,
and Proposition \ref{pr:VGIT} is proved.
\end{proof}

If $\bsalpha_0$ does not lie on any other wall,
then $\Sigmaz$ consists of one point,
and the diagram \eqref{eq:VGIT_parabolic}
is a blow-down
followed by a blow-up.
It may also happen that $\phi_+$ or $\phi_-$
is an isomorphism.

\section{Detailed description of the wall-crossing}
 \label{sc:wall-crossing_parabolic2}

Recall that $X$ is the coarse moduli space of $\bX$,
and one has a natural isomorphism
$H^0(\scO_\bX((n-3) \vecc)) \cong H^0(\scO_X(n-3))$.
Since $X$ is a projective line,
one has
$$
 \bP H^0(\scO_X(n-3)) \cong \Sym^{n-3} X \cong \bP^{n-3}.
$$
The Veronese embedding is the diagonal map
$X \hookrightarrow \Sym^{n-3} X$
sending a point $x \in X$ to $[x, \dots, x] \in \Sym^{n-3} X$.
For
$
 \veck
   = \textstyle{\sum_{i \in I} \vecx_i} + k_0 \vecc
   \in L
$
where $I = \{ i_1, \dots, i_r \} \subset \{1, \dots, n \}$ and
$k_0 \in \bZ$,
the $\veck$-th secant variety $V(\veck) \subset \Sym^{n-3} X$
is defined by
$$
 V(\veck) =
  \begin{cases}
    z_{i_1} * \cdots * z_{i_r} * \Sec_{k_0} (X) & k_0 \ge 0, \\
    \emptyset & k_0 < 0,
  \end{cases}
$$
where $X$ and the marked points $z_k \in X$ are considered
as subvarieties of $\Sym^{n-3} X$
by the Veronese embedding.
Here, the join $A * B$ of two subvarieties
of a projective space is the union
$
 \bigcup_{a \in A, \ b \in B} \ell_{a, b}
$
of lines $\ell_{a, b}$
passing through points $a \in A$ and $b \in B$,
and the $k_0$-th secant variety $\Sec_{k_0}(X) = X * \dots * X$
is the join of $k_0$ copies of $X$.

Let $I = \{ i_1, \dots, i_r \}$ be a subset of $\{ 1, \dots, n \}$ and
$J = \{ j_1, \dots, j_{n-r} \} = \{ 1, \dots, n \} \setminus I$
be its complement.
Assume that
one has
$
 - \sum_{i \in I} \alpha_{+,i} + \sum_{j \in J} \alpha_{+,j} - k < 0
$
and
$
 - \sum_{i \in I} \alpha_{-,i} + \sum_{j \in J} \alpha_{-,j} - k > 0.
$
If a vector bundle $\scP$ admits a non-trivial homomorphism
from the line bundle
$$
 \scL = \scO \lb -\vecs + \textstyle{\sum_{j \in J}} \vecx_j - k \vecc \rb,
  \quad 
 \pardeg_{\alpha} \scL
  = - \sum_{i \in I} \alpha_i + \sum_{j \in J} \alpha_j - k,
$$
then its saturation destabilizes the bundle $\scP$
with respect to the stability parameter $\bsalpha_-$.
Assume that $\scP$ is given as an extension
$$
 0
  \to \scO(-\vecs+\vecx_n)
  \to \scP
  \to \scO(-\vecx_n)
  \to 0
$$
classified by an element
$$
 e_\scP
  \in \Ext^1(\scO(-\vecx_n), \scO(-\vecs+\vecx_n))
  \cong H^0(\scO((n-3)\vecc))^\vee,
$$
and $\scO(-\vecs+\vecx_n)$ does not destabilize $\scP$
with respect to the stability parameter $\bsalpha_-$.
Then one has
$\Hom(\scL, \scO(-\vecs+\vecx_n)) = 0$ and
the $\bsalpha_-$-destabilizing morphism
$
 \scL \to \scP
$
must come from a non-trivial morphism
$
 \scL \to \scO(-\vecx_n).
$
Conversely,
a non-trivial morphism
$
 \phi \in \Hom(\scL, \scO(-\vecx_n))
$
lifts to a non-trivial morphism
$
 \varphi \in \Hom(\scL, \scP)
$
if and only if
$
 e_\scP \circ \varphi \in \Ext^1(\scL, \scO(-\vecs+\vecx_n))
$
vanishes.
Under the isomorphisms
\begin{align*}
 \Hom(\scL, \scO(-\vecx_n))
  &= H^0 \lb \scO \lb \textstyle{\sum_{i \in I}} \vecx_i
         - \vecx_n + k \vecc \rb \rb, \\
 \Ext^1(\scO(-\vecx_n), \scO(-\vecs+\vecx_n))
  &\cong H^0(\scO((n-3)\vecc))^\vee, \\
 \Ext^1(\scL, \scO(-\vecs+\vecx_n))
  &\cong H^0(\scO((n-3)\vecc
  -\lb \textstyle{\sum_{i \in I}} \vecx_i - \vecx_n + k \vecc \rb ))^\vee,
\end{align*}
the Yoneda product
$$
 \Hom(\scL, \scO(-\vecx_n))
  \otimes \Ext^1(\scO(-\vecx_n), \scO(-\vecs+\vecx_n))
  \to \Ext^1(\scL, \scO(-\vecs+\vecx_n)) 
$$
corresponds to the composition
$$
 H^0 \lb \scO \lb (n - 3) \vecc -  \lb
  \textstyle{\sum_{i \in I}} \vecx_i - \vecx_n + k \vecc \rb \rb \rb
  \otimes
 H^0 \lb \scO \lb \textstyle{\sum_{i \in I}} \vecx_i - \vecx_n + k \vecc \rb \rb
  \to H^0(\scO((n-3)\vecc)),
$$
so that there is a non-trivial morphism
$\scL \to \scP$
if and only if
$$
 [e_\scP]
  \in  \bP H^0(\scO_\bX(n-3))
  \cong \Sym^{n-3} X
  \cong \bP^{n-3}
$$
belongs to the
secant variety
$V \lb \sum_{i \in I} \vecx_i - \vecx_n + k \vecc \rb$.

\begin{remark} \label{rm:Bauer}
Bauer uses a different parametrization
of the space of stability parameters,
and the stability parameter
that he has chosen as the starting point is written as
\begin{align} \label{eq:Bauer-weight}
 \bsalpha =
  \begin{cases}
   \lb \dfrac{1}{2n-2}, \dfrac{n-2}{2n-2}, \dots, \dfrac{n-2}{2n-2} \rb
    & \text{$n$ is even}, \\
   \lb \dfrac{n-2}{2n-2}, \dots, \dfrac{n-2}{2n-2} \rb
    & \text{$n$ is odd}
  \end{cases}
\end{align}
in the notation here,
which does not satisfy $|\bsalpha|<1$.
The advantage of this stability parameter
is that the underlying bundle of a stable parabolic bundle
is always given by
\begin{align}
 \scE \cong 
  \begin{cases}
    \scO(-n/2) \oplus \scO(-n/2),  & \text{if $n$ is even,}\\
    \scO(-(n+1)/2) \oplus \scO(-(n-1)/2),
      \quad & \text{if $n$ is odd.}
  \end{cases}
\end{align}
For example,
if $n$ is even and the underlying bundle is
$\scO(-n/2-k) \oplus \scO(-n/2+k)$
for some $k > 0$,
then the parabolic degree of the subbundle
$
 \scO(-n/2+k)
$
satisfies
\begin{align*}
 \pardeg \scO(-n/2+k)
  &\ge \deg \scO(-n/2+k) + \sum_{j=1}^n \alpha_j \\
  &=  -n/2+k + \frac{1}{2n-2}+(n-1) \frac{n-2}{2n-2} \\
  &= k-1+\frac{1}{2n-2} > 0.
\end{align*}
\end{remark}

The discussion so far can be summarized
as Theorem \ref{th:VGIT_parabolic} below,
which is a variation of \cite[Theorem 2.9]{Bauer_PB}.
For the sake of simplicity of the exposition,
we restrict ourselves to the case $|\bsalpha| < 1$,
which is the case of interest for the purpose of this paper;
this allows us to deal only with walls $H_{I, k}$ with $k = 0$.

\begin{theorem} \label{th:VGIT_parabolic}
The moduli space $\scN_\bsalpha$
for any parameter $\bsalpha = (\alpha_1, \dots, \alpha_n)$
satisfying $|\bsalpha| = \alpha_1 + \dots + \alpha_n < 1$
is described as follows:
\begin{enumerate}
\setcounter{enumi}{-1}
 \item
Assume $\alpha_1 \le \alpha_2 \le \dots \le \alpha_n$
by reordering the points if necessary.
Set
$
 \bsbeta_0
  = (r \alpha_1, \dots, r \alpha_{n-1}, \alpha_n)
$
for a sufficiently small positive number $r$,
so that $\bsbeta_0$ belongs to the chamber
described in Proposition \ref{pr:distinguished_parameter} and
one has $\scN_{\bsbeta_0} \cong \Sym^{n-3} X \cong \bP^{n-3}$.
 \item
We first cross walls of the form $H_{\{i, n\}, 0}$
for $1 \le i \le n-1$ satisfying
\begin{equation} \label{eq:wall1}
 |\bsalpha| - 2 \alpha_i - 2 \alpha_n > 0.
\end{equation}
When we cross the wall $H_{\{i,n\},0}$,
the moduli space is blown-up
at the point $z_i \in X \subset \Sym^{n-3} X \cong \bP^{n-3}$.
After crossing all these walls,
we arrive at the stability parameter $\bsbeta_1$
such that $\scN_{\bsbeta_1}$ is obtained from
$\scN_{\bsbeta_0}$
by blowing up the points $z_i$
for $1 \le i \le n-1$ satisfying \eqref{eq:wall1}.
 \item
We then cross walls of the form $H_{\{i_1, i_2, n\}, 0}$
for $1 \le i_1 < i_2 \le n-1$ satisfying
$$
 |\bsalpha| - 2 \alpha_{i_1} - 2 \alpha_{i_1} - 2 \alpha_n > 0.
$$
When we cross the wall $H_{\{i_1, i_2, n\}, 0}$,
the moduli space is blown-down
along the strict transform
of the line passing through $z_{i_1}$ and $z_{i_2}$,
and then blown-up in the other direction
so that the exceptional divisor is isomorphic to $\bP^{n-5}$.
In other words,
we blow-up the moduli space
along the strict transform
of the line passing through $z_{i_1}$ and $z_{i_2}$,
and contract it down to the other direction.
 \item
In the $r$-th step,
we cross the walls $H_{\{i_1, \dots, i_r, n\}, 0}$
for $1 \le i_1 < \dots < i_r \le n-1$ satisfying
$$
 |\bsalpha| - 2 \alpha_{i_1} - \dots - 2 \alpha_{i_r} - 2 \alpha_n > 0.
$$
Note that this condition can be written as
$$
 \alpha_{i_1} + \dots + \alpha_{i_r} + \alpha_n
  < \alpha_{j_1} + \dots + \alpha_{j_{n-r-1}}
$$
where $\{ j_1, \dots, j_{n-r-1} \}$ is the complement
of $\{ i_1, \dots, i_r, n \}$ in $\{ 1, \dots, n \}$.
When we cross the wall $H_{\{i_1, \dots, i_r, n\}, 0}$,
the moduli space is blown-up
along the strict transform of the $(r-1)$-dimensional linear subspace
spanned by $z_{i_1}, \dots, z_{i_r}$,
and then contracted to the other direction.
This is a birational transformation
which replaces $\bP^{r-1}$ with $\bP^{n-r-4}$.
 \item
By successively crossing the walls as above,
we arrive at the chamber containing $\bsalpha$.
\end{enumerate}
\end{theorem}

\section{Wall crossing in $\scM_\bsw$} \label{sc:comparison}

Let $\bsw = (w_1, \ldots, w_n) \in \bQ^n$
be a stability parameter
for the moduli space of ordered $n$-points on $\bP^1$,
which can be taken from
$$
W = \lc \left. \bsw = (w_1, \dots, w_n) \in \bQ^n \, \right| \,
  |\bsw| = w_1 + \dots + w_n = 2 \rc
$$
by rescaling $\bsw$ if necessary;
unlike the moduli space $\scN_\bsalpha$,
the overall rescaling of $\bsw$ only changes
the ample $\bQ$-line bundle on $\scM_\bsw$ and
does not affect the moduli space $\scM_\bsw$.
A configuration
$(x_1, \ldots, x_n)$
of ordered $n$ points on $\bP^1$ is $\bsw$-semistable
if for any point $x \in \bP^1$,
one has
$$
 \sum_{i=1}^n \delta_{x, x_i} w_i \le 1.
$$
The moduli space $\scM_\bsw$ contains the configuration space
$$
 X(2,n) = ((\bP^1)^n \setminus \Delta) / \PGL_2
$$
of $n$ points on $\bP^1$ as an open subscheme
if and only if $\bsw \in (0,1)^n$,
where
$$
 \Delta = \{ (x_1, \dots, x_n) \in (\bP^1)^n \mid
  x_i = x_j \text{ for some } i \ne j \}
$$
is the big diagonal.
By normalizing the last three points as
$
 (x_{n-2}, x_{n-1}, x_n) = (0, 1, \infty)
$
by the $\PGL_2$-action,
one can realize $X(2,n)$ as an open subscheme
$$
 X(2,n)
  \cong \{ [x_1:\dots:x_{n-3}:1] \in \bP^{n-3} \mid
   x_i \ne 0, 1, x_j \text{ for }i \ne j \},
$$
which is the complement
of a hyperplane arrangement in $\bP^{n-3}$.

Walls in the space $W$ of stability parameters are given by
$$
 H_I = \lc \bsw \in W \, \left| \, \textstyle{\sum_{i \in I}} w_i = 1
   \right. \rc
$$
for a proper subset $I = \{ i_1, \dots, i_r \}$ of $\{ 1, \dots, n \}$.
Note that $\sum_{i \in I} w_i = 1$ implies
$\sum_{j \in J} w_j = 1$
for $J = \{j_1, \dots, j_{n-r} \} = \{ 1, \dots, n \} \setminus I$.
Let $C_+$ and $C_-$ be two chambers
separated by the wall $W_I$, and
take stability conditions
$\bsw_+ \in C_+$, $\bsw_- \in C_-$, and
$\bsw_0 \in W_I$.
Assume that $\sum_{i \in I} w_{i, +} > 1$,
$\sum_{i \in I} w_{i, -} < 1$, and
$\bsw_0$ is not on any other walls.
Then one has a diagram
\begin{equation} \label{eq:VGIT_points}
\begin{array}{c}
\begin{psmatrix}[colsep=1cm, rowsep=1cm]
 \scMp & & \scMm \\
 & \scMz
\end{psmatrix}
\psset{shortput=nab,labelsep=1pt,nodesep=3pt,arrows=->}
\ncline{1,1}{2,2}_{\phi_+}
\ncline{1,3}{2,2}^{\phi_-}
\end{array}
\end{equation}
where $\phi_+$ blows-down the subvariety
$$
 S_{\bsw_+}
  = \{ x_{j_1} = \dots = x_{j_{n-r}} \}
  \cong \bP^{r-2}
$$
of $\scMp$
to the subvariety
$$
 S_{\bsw_0}
  = \lc x_{i_1} = \dots = x_{i_r}, \ 
      x_{j_1} = \dots = x_{j_{n-r}} \rc
$$
of $\scMz$ consisting of just one point,
and $\phi_-$ blows-down the subvariety
$$
 S_{\bsw_-}
  = \{ x_{i_1} = \dots = x_{i_r} \}
  \cong \bP^{n-r-2}
$$
of $\scMm$
to the same point in $\scMz$.
This wall-crossing is described in
\cite[Theorem 4.1]{MR3194079}
from a symplectic point of view.

The diagram \eqref{eq:VGIT_points}
for the special case $I = \{ n \}$ gives a wall-crossing
from the empty space $\scMp = S_{\bsw_+} \cong \bP^{-1} = \emptyset$
to the projective space
$\scMm = S_{\bsw_-} \cong \bP^{n-3}$
through one point $\scMz = S_{\bsw_0}$.
The chamber $C_-$ containing $\bsw_-$ in this case
is defined by
\begin{equation} \label{eq:C1}
 C_- = \{ \bsw \in W \mid w_n < 1 \text{ and }
  w_i + w_n > 1 \text{ for any } 1 \le i \le n - 1 \}.
\end{equation}
The moduli space $\scM_\bsw$ for $\bsw \in C_-$
is described explicitly as follows:
One can set $x_n = \infty \in \bP^1$
by the $\PGL_2$-action.
Since one has $x_i \ne x_n$ for any $1 \le i \le n-1$
by the stability condition,
one must have $(x_1, \dots, x_{n-1}) \in \bA^{n-1}$.
One can set $x_{n-1} = 0$ by the residual $\PGL_2$-action,
and then one is left with the $\bG_m$-action on $\bA^{n-2}$.
The stability condition prohibits $x_1 = \dots = x_{n-1}$,
so that one cannot have $x_1 = \dots = x_{n-2} = 0$.
This shows that one has
$$
 \scM_\bsw = (\bA^{n-2} \setminus \{ 0 \}) / \bG_m,
$$
which is nothing but the projective space $\bP^{n-3}$.

\begin{theorem} \label{th:VGIT_points}
The moduli space $\scM_\bsw$
for any stability parameter $\bsw = (w_1, \dots, w_n)$
can be obtained from $\bP^{n-3}$
by the following birational transformations:
Assume $w_1 \le w_2 \le \dots \le w_n$
by reordering the points if necessary.
We start from the chamber \eqref{eq:C1} and
gradually increase $w_1, \dots, w_{n-1}$ and
decrease $w_n$.
Set $p_i = [\delta_{i0}:\dots:\delta_{i, n-2}] \in \bP^{n-3}$
for $1 \le i \le n-2$
and $p_{n-1} = [1:\dots:1] \in \bP^{n-3}$.
\begin{enumerate}
 \item
We first cross the walls $H_{\{i, n\}}$
for $1 \le i \le n-1$ satisfying $w_i + w_n < 1$.
When we cross the wall $H_{\{i,n\}}$,
the moduli space is blown-up
at the point $p_i$.
 \item
We then cross the walls $H_{\{i_1, i_2, n\}}$
for $1 \le i_1 < i_2 \le n-1$ satisfying $w_{i_1} + w_{i_1} + w_n < 1$.
When we cross the wall $H_{\{i_1, i_2, n\}}$,
the moduli space is blown-down
along the strict transform
of the line passing through $p_{i_1}$ and $p_{i_2}$,
and then blown-up in the other direction,
so that the exceptional divisor is isomorphic to $\bP^{n-5}$.
In other words,
we blow-up the moduli space
along the strict transform
of the line passing through $p_{i_1}$ and $p_{i_2}$,
and contract it down to the other direction.
 \item
In the $r$-th step,
we cross the walls $H_{\{i_1, \dots, i_r, n\}}$
for $1 \le i_1 < \dots < i_r \le n-1$
satisfying $w_{i_1} + \dots + w_{i_r} + w_n < 1$.
Note that this condition is equivalent to
$w_{i_1} + \dots + w_{i_r} + w_n < w_{j_1} + \dots + w_{j_{n-r-1}}$
where $\{ j_1, \dots, j_{n-r-1} \}$ is the complement
of $\{ i_1, \dots, i_r, n \}$ in $\{ 1, \dots, n \}$.
When we cross the wall $H_{\{i_1, \dots, i_r, n\}}$,
the moduli space is blown-up
along the strict transform of the $(r-1)$-dimensional linear subspace
spanned by $p_{i_1}, \dots, p_{i_r}$,
and then contracted down to the other direction.
This is a birational transformation
which replaces $\bP^{r-1}$ with $\bP^{n-r-4}$.
 \item
By successively crossing the walls as above,
we arrive at the chamber containing $\bsw$.
\end{enumerate}
\end{theorem}

\begin{example}
Set $n = 5$ and
$$
 \bsw_1 = \lb \frac{2}{5}, \frac{2}{5}, \frac{2}{5}, \frac{2}{5}, \frac{2}{5} \rb.
$$
We consider the straight line segment
$$
 \bsw_t = (1-t) \bsw_0 + t \bsw_1
$$
starting from the stability parameter
$$
 \bsw_0 = \lb \frac{5}{16}, \frac{5}{16}, \frac{5}{16}, \frac{5}{16}, \frac{3}{4} \rb
$$
in the chamber \eqref{eq:C1}
satisfying $\scM_{\bsw_0} \cong \bP^2$.
The wall-crossing takes place
at $t = \frac{5}{21}$ and
$$
 \bsw_t = \lb \frac{1}{3}, \frac{1}{3}, \frac{1}{3}, \frac{1}{3}, \frac{2}{3} \rb,
$$
where the points $x_i = x_j = x_k$
for $1 \le i < j < k \le 4$ are stable for $t \le \frac{5}{21}$
and unstable for $t > \frac{5}{21}$.
These points are blown-up
by the wall-crossing,
so that the point $x_i = x_j = x_k$ is replaced
by the exceptional divisor
$x_\ell = x_5$
where $\{ i, j, k, \ell \} = \{ 1, 2, 3, 4 \}$.
With respect to the normalization
$$
 (x_1, x_2, x_3, x_4, x_5)
  = (x, y, 1, 0, \infty),
$$
four points at the center of the blow-up are given by
\begin{align*}
 x_1 = x_2 = x_3 \ : \ [x:y:1] = [1:1:1] \in \bP^2, \\
 x_1 = x_2 = x_4 \ : \ [x:y:1] = [0:0:1] \in \bP^2, \\
 x_1 = x_3 = x_4 \ : \ [x:y:1] = [0:1:0] \in \bP^2, \\
 x_2 = x_3 = x_4 \ : \ [x:y:1] = [1:0:0] \in \bP^2.
\end{align*}
These points are in general position,
so that $\scM_{\bsw_1}$ is $\bP^2$
blown-up at four points in general position.
\end{example}

Now we are ready to prove the following:

\begin{theorem} \label{th:comparison}
Let $\bsalpha$ be a stability parameter 
for the moduli space of parabolic bundles
satisfying $|\bsalpha| < 1$, and
$\bsw = 2 \bsalpha / |\bsalpha|$ be the corresponding
normalized stability parameter
for the moduli space ordered $n$ points on $\bP^1$.
Then one has an isomorphism
$$
 \scN_\bsalpha \cong \scM_\bsw
$$
of algebraic varieties.
\end{theorem}

\begin{proof}
Since the wall-crossing in $\scN_\bsalpha$ and $\scM_\bsw$
described in Theorem \ref{th:VGIT_parabolic}
and Theorem \ref{th:VGIT_points} are identical,
it suffices to show the existence of an isomorphism
$$
 \scN_\bsalpha \cong \scM_\bsw
$$
for a stability parameter $\bsalpha$
satisfying the condition in Proposition
\ref{pr:distinguished_parameter},
such that the points $z_i \in \scN_\bsalpha$ are mapped to
$p_i \in \scM_\bsw$ for $i = 1, \dots, n-1$.
This is clear, since both moduli spaces are
$(n-3)$-dimensional projective spaces and
the points are $n-1$ points in general position.
\end{proof}

A more general result,
which gives an isomorphism
between the moduli space of parabolic $G$-bundles
for a simply-connected simple algebraic group $G$
and a GIT quotient of a product of flag varieties,
is shown in \cite[Proposition 4.8]{0910.0577}.


\section{Bending systems on $\scM_\bsw$} \label{sc:bending}

Let $G = SU(2)$, and identify the Lie algebra $\frakg = \su$ 
with its dual by the Killing form
$
 \langle  \,\, , \,\, \rangle : \frakg \times \frakg \to \bR.
$
Let $T \subset G$ be the maximal torus consisting of diagonal matrices,
and take a base 
\[
  x_0 = \begin{pmatrix}
    2\pi \sqrt{-1} & 0 \\
    0 & -2\pi \sqrt{-1}
  \end{pmatrix}
\]
of the Lie algebra $\frakt$ of $T$.
For $\alpha \in \bR_{>0}$,
the adjoint orbit $\scO_{\alpha} \subset \frakg$  of $\alpha x_0$
has a natural symplectic form called the Kostant-Kirillov form, as follows.
Recall that a tangent vector of $\scO_{\alpha}$ at $x$ can be written as
$\ad_{\xi}(x) = [x,\xi]$ for $\xi \in \frakg$.
The Kostant-Kirillov form $\omega_{\mathcal{O}_{\alpha}}$ is given by
\[
  \omega_{\mathcal{O}_{\alpha}} (\ad_{\xi}(x), \ad_{\eta}(x))
  = \langle x, [\xi, \eta] \rangle.
\]

For $\bsalpha = (\alpha_1, \dots, \alpha_n) \in (\bR_{> 0})^n$, 
we define
$\scO_{\bsalpha} = \prod_i \scO_{\alpha_i} \subset \frakg^n$
with the $i$-th projection
$\mathrm{pr}_i : \scO_{\bsalpha} \to \scO_{\alpha_i}$, $i=1, \dots, n$.
The diagonal $G$-action on $\scO_{\bsalpha}$
is Hamiltonian with respect to the symplectic form 
$\sum_i \mathrm{pr}_i^* \omega_{\mathcal{O}_{\alpha_i}}$, and 
its moment map is given by
\[
  \mu : \scO_{\bsalpha}
  \longrightarrow  \frakg, \quad
  \bsx = (x_1, \dots, x_n) \longmapsto x_1 + \dots + x_n.
\]
From the Kirwan-Kempf-Ness Theorem, the symplectic reduction
\begin{equation}
  \mu^{-1}(0)/G
  = \{ \bsx  \in \scO_{\bsalpha} \, 
     | \,  x_1 + \dots + x_n = 0 \, \}/ G
  \label{eq:PolygonSp_MWred}
\end{equation}
is diffeomorphic to $\scM_{\bsw}$ for $\bsw = 2 \bsalpha /|\bsalpha|$,
and the induced symplectic form 
is compatible with the complex structure
(on the smooth locus of $\scM_{\bsw}$).
In what follows we write this space as $\scM_{\bsalpha}$ 
to emphasize its symplectic structure
$\omega_{\scM_{\bsalpha}}$.
Note that $(\scM_{k \bsalpha}, \omega_{\scM_{k \bsalpha}})$ is
symplectomorphic to $(\scM_{\bsalpha}, k \omega_{\scM_{\bsalpha}})$
for $k > 0$.
The expression \eqref{eq:PolygonSp_MWred} shows that 
$\scM_{\bsalpha}$ parametrizes $n$-gons in 
$\frakg \cong \bR^3$ with fixed side lengths $\alpha_1, \dots, \alpha_n$
modulo Euclidean motions.

Let $e_1, \dots, e_n \in \bR^2$ denote side edge vectors
of a reference $n$-gon $P \subset \bR^2$,
satisfying $e_1 + \dots + e_n = 0$.
For a diagonal $d = e_i + e_{i+1} + \dots + e_{i+k}$ of $P$, 
we define $\varphi_{d} : \scM_{\bsalpha} \to \mathbb{R}$ 
as the length function
\[
  \varphi_{d}(\bsx) = |x_{i} + x_{i+1} + \dots + x_{i+k}|
\]
of the corresponding diagonal in $\bsx$.
This function is called a {\it bending Hamiltonian},
since the Hamiltonian flow of $\varphi_d$ bends $n$-gons around 
the diagonal corresponding to $d$ 
(see \cite{Kapovich-Millson_SGPES} or \cite{Klyachko_SP}).

We fix a triangulation of $P$ given by $n-3$ diagonals 
$d_1, \dots, d_{n-3}$ which 
do not intersect in the interior of $P$,
and let $\Gamma$ denote its dual graph.
Note that $\Gamma$ is a trivalent tree with $n$ leaves.
The bending system associated to $\Gamma$ is defined by
\[
  \Phi_{\Gamma} = ( \varphi_{d_1}, \dots, \varphi_{d_{n-3}}) : 
  \scM_{\bsalpha} \longrightarrow \mathbb{R}^{n-3}.
\]

\begin{theorem}[Kapovich and Millson \cite{Kapovich-Millson_SGPES},  
Klyachko \cite{Klyachko_SP}]
The $(n-3)$-tuple of functions $\Phi_{\Gamma}$
is a completely integrable system on $\scM_{\bsalpha}$.
The functions $\varphi_{d_i}$ are action variables, and hence 
define a Hamiltonian torus action on an open dense subset
where $\varphi_{d_i}$ are smooth.
The image 
$$
 \Delta_\Gamma(\bsalpha)
  := \Phi_\Gamma(\scM_{\bsalpha})
  \subset \bR^{n-3}
$$
is a convex polytope defined by triangle inequalities.
\end{theorem}


\section{Goldman systems on $\scN_{\bsalpha}$} \label{sc:Goldman_system}

Let $(X, (z_1, \ldots, z_n))$ be a projective line
with $n$ marked points.
For each marked point $z_i \in X$,
we take a small open disk $D_i \subset X$ around $z_i$ such that
$\overline{D_i} \cap \overline{D_j} = \emptyset$ for $i \ne j$, and
set $\RS = X \setminus (D_1 \cup \dots \cup D_n)$.
Then the fundamental group of $\RS$ is given by 
\[
  \pi_1 (\RS) = \langle
  \gamma_1, \gamma_2, \dots, \gamma_n \, | \,
  \gamma_1 \dots \gamma_n =1 \rangle,
\]  
where $\gamma_i$ is the homotopy class representing the $i$-th
boundary component $\partial D_i$.

For $\bsalpha = (\alpha_1, \dots, \alpha_n) \in (0, 1/2)^n$,
let $\scC_{\alpha_i} \subset G$ denote the conjugacy class of 
$e^{\alpha_i x_0} = 
\diag ( e^{2\pi \sqrt{-1} \alpha_i}, e^{-2\pi \sqrt{-1} \alpha_i} )$,
and set $\scC_{\bsalpha} = \prod_{i=1}^n \scC_{\alpha_i} \subset G^n$.
As recalled in \pref{sc:Mehta-Seshadri},
the moduli space of parabolic $SU(2)$-bundles on $X$
with parabolic weight $\bsalpha$
can be identified with the moduli space
\begin{align*}
  \scN_{\bsalpha} (\RS)
  &:= \{ \rho \in \Hom (\pi_1(\RS), G) \, | \, 
       \rho (\gamma_i) \in \scC_{\alpha_i}, \,\, i=1, \dots, n \}/G \\
  &\cong \{ \bsg = (g_1, \dots, g_n) \in \scC_{\bsalpha} \, | \,
      g_1 \dots g_n = 1 \}/G
\end{align*}
of $G$-representations of $\pi_1(\RS)$.
Since $\scC_{\alpha_i}$ is a geodesic sphere around the identity,
$\scNa (\RS)$ is regarded as a moduli space of $n$-gons
in $G \cong S^3$ with fixed side lengths
(cf.~e.g.~\cite{Millson-Poritz}).


We recall the description of the symplectic structure on $\scNa (\RS)$
given in \cite{Guruprasad-Huebschmann-Jeffrey-Weinstein}.
Fix a representation $\rho$ in
$$
  \widetilde{\scN}_{\bsalpha} = \{ \rho \in \Hom(\pi_1 (\RS), G) \mid
   \rho(\gamma_i ) \in \scC_{\alpha_i}, \, i=1, \dots, n
   \}
$$
and let $\frakg_{\rho}$ denote the representation of $\pi_1(\RS)$
on $\frakg$ given by
$$
  \pi_1(\RS) \overset{\rho}{\longrightarrow}
  G \overset{\Ad}{\longrightarrow}
  \Aut (\frakg).
$$ 
Take a curve $\rho_t$ in $\widetilde{\scN}_{\bsalpha}$ with $\rho_0 = \rho$
and set
$
 u= \left.\frac{d}{dt}\right|_{t=0} \rho_t \colon \pi_1(\RS) \to \frakg.
$
Then $\rho_t$ can be written as
$$
  \rho_t (\gamma ) =
    \exp (t u(\gamma) + O(t^2)) \rho (\gamma).
$$
The homomorphism condition 
$\rho_t(\gamma \gamma') 
  = \rho_t(\gamma) \rho_t (\gamma')$ implies that
\begin{equation}
  u (\gamma \gamma') = u(\gamma) +
    \Ad_{\rho (\gamma)} u(\gamma').
  \label{eq:cocycle1}
\end{equation}
From the boundary condition
$\rho_t (\gamma_i) \in \scC_{\alpha_i}$, we have
$\rho_t (\gamma_i) = g_{i,t}^{-1} \rho(\gamma_i) g_{i,t}$
for some $g_{i,t} \in G$.
This implies that 
\begin{equation}
  u(\gamma_i) = \Ad_{\rho(\gamma_i)} \xi_i - \xi_i
  \label{eq:cocycle2}
\end{equation}
for each $i$, where 
$\xi_i = \left. \frac{d}{dt} \right|_{t=0} g_{i,t} \in \frakg$.
Namely, $T_{\rho} \widetilde{\scN}_{\bsalpha}$ is identified
with the space of {\it parabolic 1-cocycles}:
\begin{align*}
  T_{\rho} \widetilde{\scN}_{\bsalpha} &\cong 
    Z^1_{\mathrm{par}} (\pi_1(\RS); \frakg_{\rho}) \\
    &= \{ u: \pi_1(\RS) \to \frakg \, | \,
         \text{$u$ satisfies 
         (\ref{eq:cocycle1}) and (\ref{eq:cocycle2})}
        \, \}.
\end{align*}
Similarly, the tangent space to the $G$-orbit of $\rho$ 
is spanned by {\it parabolic 1-coboundaries}
$$
  u(\gamma) = \Ad_{\rho(\gamma)} \xi - \xi, \quad \xi \in \frakg.
$$
Let $B^1_{\mathrm{par}} (\pi_1(\RS); \frakg_{\rho})$
denote the vector space of parabolic 1-coboundaries.
Then the tangent space $T_{\rho} \scNa$ at $\rho$ is identified
with the {\it first parabolic cohomology}
$$
 \Hpar^1(\pi_1(\RS); \frakg_\rho)
  = Z^1_{\mathrm{par}} (\pi_1(\RS); \frakg_{\rho})
      / B^1_{\mathrm{par}} (\pi_1(\RS); \frakg_{\rho}).
$$
The space of 2-chains 
$C_2(\pi_1(\RS); \mathbb{Z})$ is generated by symbols
$[\gamma|\gamma']$ for $\gamma, \gamma' \in \pi_1(\RS)$,
and the cup product
$$
  \cup : H^1_{\mathrm{par}}(\pi_1(\RS); \frakg_{\rho})
  \times H^1_{\mathrm{par}}(\pi_1(\RS); \frakg_{\rho})
  \longrightarrow 
  H^2 (\pi_1(\RS), \partial \pi_1(\RS) ; \mathbb{R})
$$
is given by 
$$
  (u \cup v) ([\gamma | \gamma']) 
   = \langle u(\gamma), \Ad_{\rho(\gamma)} v(\gamma')
     \rangle
$$
for 1-cocycles $u$, $v$.
In what follows we write $\Ad_{\gamma} = \Ad_{\rho (\gamma)}$
for short.
The relative fundamental class in
$H_2(\pi_1(\RS), \partial \pi_1(\RS); \mathbb{Z})$ is represented by
$$
  [\pi_1(\RS), \partial \pi_1(\RS)] = 
  \sum_{i=1}^{n-1} [\gamma_1 \dots \gamma_i | \gamma_{i+1}].
$$

\begin{theorem}[Guruprasad et al.
{\cite[Key Lemma 8.4]{Guruprasad-Huebschmann-Jeffrey-Weinstein}}]
  Let $u$, $v$ be parabolic 1-cocycles such that
  $u(\gamma_i) = \Ad_{\gamma_i} \xi_i - \xi_i$ and 
  $v(\gamma_i) = \Ad_{\gamma_i} \eta_i - \eta_i$,
  $i=1, \dots, n$,  
  respectively.
  Then the symplectic form on $\scNa (\RS)$ is given by
  \begin{equation}
    \omega_{\scNa} (u, v) 
    = (u \cup v) ( [\pi_1, \partial \pi_1] )
    + \frac 12 \sum_{i=1}^n
      (\langle \xi_i, \Ad_{\gamma_i}\eta_i \rangle
      - \langle \eta_i, \Ad_{\gamma_i}\xi_i \rangle ).
  \label{eq:symp_par}
  \end{equation}
\end{theorem}

For a later use, we write the first term of (\ref{eq:symp_par}) 
more explicitly.
By using (\ref{eq:cocycle1}) inductively, we have
$$
  u(\gamma_1 \dots \gamma_i) 
  = \sum_{k=1}^i \Ad_{\gamma_1 \dots \gamma_{k-1}} u(\gamma_k)
  = \sum_{k=1}^i \Ad_{\gamma_1 \dots \gamma_{k-1}}
    (\Ad_{\gamma_k} \xi_k - \xi_k).
$$
Hence we obtain
\begin{align*}
  (u \cup v) ( [\pi_1, \partial \pi_1] )
  &= \sum_{i=1}^{n-1} 
     \langle u( \gamma_1\dots \gamma_i),
     \Ad_{\gamma_1 \dots \gamma_i} v(\gamma_{i+1}) \rangle\\
  &= \sum_{i=1}^{n-1} \sum_{k=1}^{i}
     \langle \Ad_{\gamma_1\dots \gamma_{k-1}} u(\gamma_k),
     \Ad_{\gamma_1 \dots \gamma_i} v(\gamma_{i+1}) \rangle\\
  &= \sum_{i=1}^{n-1} \sum_{k=1}^{i}
     \langle  u(\gamma_k),
     \Ad_{\gamma_k \dots \gamma_i} v(\gamma_{i+1}) \rangle\\
  &= \sum_{i=1}^{n-1} \sum_{k=1}^{i}
     \langle  \Ad_{\gamma_k} \xi_k - \xi_k,
     \Ad_{\gamma_k \dots \gamma_i} 
     (\Ad_{\gamma_{i+1}}\eta_{i+1} - \eta_{i+1}) \rangle.
\end{align*}

Next we recall a completely integrable system on $\scN_{\bsalpha}(\RS)$
introduced by Goldman \cite{Goldman_IFLG}.
For a simple closed curve $C \subset \RS$, we write
$[C] = \gamma_i \gamma_{i+1} \dots \gamma_{i+k}$ in $\pi_1(\RS)$, and 
define a function 
$\vartheta_C = \theta_{\bsalpha, C} \colon \scN_{\bsalpha}(\RS) \to \bR$ by
\[
  \vartheta_C (\bsg)= \cos^{-1} \left( \frac 12 \mathrm{tr}
    (g_i g_{i+1} \dots g_{i+k}) \right).
\]
Take a set $C_1, \dots, C_{n-3}$ of simple closed curves
defining a pair-of-pants decomposition of $\RS$.
Note that the set of such choices is in one-to-one correspondence
with the set of trivalent trees $\Gamma$ with $n$-leaves.
We then obtain a set of $n-3$ functions
$$
  \Theta_{\bsalpha, \Gamma} =
  \Theta_\Gamma = (\vartheta_{C_1}, \dots, \vartheta_{C_{n-3}}) :
  \scNa \longrightarrow \mathbb{R}^{n-3}.
$$ 

\begin{theorem}[Goldman \cite{Goldman_IFLG}, 
Jeffrey and Weitsman \cite{Jeffrey-Weitsman_BSO}]
For each pair-of-pants decomposition of $\RS$ with dual graph $\Gamma$, 
the set of functions
$\Theta_\Gamma : \scNa \to \mathbb{R}^{n-3}$ 
is a completely integrable system.
The functions $\vartheta_{C_i}$ are action variables, and hence define a 
Hamiltonian torus action on an open dense subset of $\scN_{\bsalpha}$.
The image $\Theta_{\Gamma} (\scNa) \subset \mathbb{R}^{n-3}$ 
is a convex polytope given by the inequalities
\[
  |u_{k_1}- u_{k_2}| \le u_{k_3} 
  \le \min \{ u_{k_1}+ u_{k_2}, 
  2-( u_{k_1}+ u_{k_2}) \} 
\]
for each pair-of-pants.
In particular, if $|\bsalpha| < 1$, then the image is given by
triangle inequalities, i.e., 
$\Theta_{\Gamma} (\scNa) = \Delta_{\Gamma}(\bsalpha)$.
\end{theorem}

\section{Extended moduli spaces} \label{sc:extended_moduli}

Fix base points of $\partial D_i$ for $i=1,\ldots,n$.
Let $B_i$ for $i=1,\ldots,n$ be the loop around $\partial D_i$
starting and ending at the base point on $\partial D_i$,
and $A_i$ for $i=2,\ldots,n$ be the path
from the base point on $\partial D_i$ to the base point on $\partial D_1$.
Then the generators of $\pi_1(\RS)$ are given by 
$\gamma_1 = [B_1], \gamma_2 = [A_2 B_2 A_2^{-1}], \dots, 
\gamma_n = [A_n B_n A_n^{-1}]$. 
Let 
\[
  \alcove = \{ \alpha x_0 \in \frakt \, | \, \alpha \in [0,1/2] \}
  \subset \frakt
\]
denote the fundamental alcove.

\begin{definition}[Jeffrey \cite{Jeffrey_EMS}, 
Hurtubise and Jeffrey {\cite[Section 2]{Hurtubise-Jeffrey_RWFFPB}}]
The {\it $G$-extended moduli space} $\scN^G(\RS)$ is the space of 
$G$-representations of the groupoid generated by $A_2, \dots, A_n$ and
$B_1, \dots, B_n$, or equivalently,
\[ 
  \scN^{G}(\RS) = \{ (\bsa, \bsb)
                     \in G^{n-1} \times G^n\, | \,
                     b_1 (a_2 b_2 a_2^{-1}) \dots (a_n b_n a_n^{-1}) = 1
                  \},
\]
where $(\bsa,\bsb)=(a_2, \dots, a_n, b_1, \dots, b_n)$.
The {\it $T$-extended moduli space} is defined by
\[ 
  \scN^{T}(\RS) = \{ (\bsa, \bsb)
                     \in \scN^G(\RS) \, | \,
                     b_i \in \exp(\alcove), \, \, i=1, \dots, n
                  \} \subset G^{n-1} \times T^n.
\]
The  $\frakg$- and {\it $\frakt$-extended moduli spaces} are defined by
\begin{align*}
  \scN^{\frakg}(\RS) 
  &= \{ (\bsa, \bsx) \in G^{n-1} \times \frakg^n\, | \,
                     e^{x_1} (a_2 e^{x_2} a_2^{-1}) 
                     \dots (a_n e^{x_n}  a_n^{-1}) = 1
                  \}, \\
  \scN^{\frakt}(\RS) 
  &= \{ (\bsa, \bsx) \in G^{n-1} \times \frakt^n\, | \,
                     e^{x_1} (a_2 e^{x_2} a_2^{-1}) 
                     \dots (a_n e^{x_n}  a_n^{-1}) = 1
                  \},
\end{align*}
respectively, 
where $(\bsa, \bsx) = (a_2, \dots, a_n, x_1,  \dots , x_n)$.
\end{definition}

Each $a_i$ and $b_i$ are regarded as holonomies of a flat parabolic connection
along $A_i$ and $B_i$, respectively.
Note that we have a natural surjection
$\scN^{\frakg}(\RS) \to \scN^{G}(\RS)$ given by
\[
  (a_2, \dots, a_n, x_1,  \dots , x_n) \longmapsto
  (a_2, \dots, a_n, e^{x_1},  \dots , e^{x_n}).
\]
On the other hand, $\scN^{T}(\RS)$ is canonically embedded into 
$\scN^{\frakt}(\RS)$ by
\[
  (a_2, \dots, a_n, e^{x_1},  \dots , e^{x_n}) \longmapsto
  (a_2, \dots, a_n, x_1,  \dots , x_n).
\]

\begin{proposition}[{\cite[Propositions 2.11 and 2.12]
{Hurtubise-Jeffrey_RWFFPB}}]
The space $\scN^{G}(\RS)$ is diffeomorphic to $G^{2(n-1)}$ by
\[
 \scN^{G}(\RS) \to G^{2(n-1)},  \quad 
  (a_2, \dots, a_n, b_1, b_2, \dots, b_n) \longmapsto
  (a_2, \dots, a_n, b_2, \dots, b_n),
\]
and hence it is smooth.
On the other hand, 
$\scN^{\frakg}(\RS)$ is smooth outside the subset consisting of 
$(\bsa, \bsx)$ satisfying $e^{x_i}= -1$ for all $i$.
\end{proposition}

The group $G^n$ acts on $\scN^{G}(\RS)$ and $\scN^{\frakg}(\RS)$ by
\begin{align*}
 \bssigma \cdot (\bsa, \bsb)
  &= (\sigma_1 a_2 \sigma_2^{-1}, \dots, \sigma_1 a_n \sigma_n^{-1}, 
     \sigma_1 b_1 \sigma_1^{-1}, \dots , \sigma_n b_n \sigma_n^{-1}), \\
 \bssigma \cdot (\bsa, \bsx)
  &= (\sigma_1 a_2 \sigma_2^{-1},  \dots, \sigma_1 a_n \sigma_n^{-1}, 
     \Ad_{\sigma_1} x_1,  \dots , \Ad_{\sigma_n} x_n),
\end{align*}
for $\bssigma = (\sigma_1, \dots, \sigma_n) \in G^n$. 
These actions induce $T^n$-actions 
\begin{align*}
  \bssigma \cdot (\bsa, \bsb)
  &= (\sigma_1 a_2 \sigma_2^{-1}, \dots, \sigma_1 a_n \sigma_n^{-1}, 
      b_1, \dots ,  b_n ), \\
  \bssigma \cdot (\bsa, \bsx)
  &= (\sigma_1 a_2 \sigma_2^{-1},  \dots, \sigma_1 a_n \sigma_n^{-1}, 
      x_1,  \dots ,  x_n)
\end{align*}
on $\scN^{T}(\RS)$ and $\scN^{\frakt}(\RS)$, respectively.

\begin{proposition}[\cite{Jeffrey_EMS}, 
{\cite[Proposition 2.14]{Hurtubise-Jeffrey_RWFFPB}}]
There exists a closed two-form on $\scN^{\frakg}(\RS)$
which is non-degenerate on an open dense subset,
and for which the map
\[
  \mu^{\frakg} : \scN^{\frakg} (\RS) \longrightarrow {\frakg}^n,
  \quad (\bsa,\bsx) \longmapsto - \bsx
  = ( -x_1, \dots, -x_n)
\]
is the moment map of the $G^n$-action. 
The symplectic reduction 
$\bigl(\mu^{\frakg}\bigr)^{-1} (\scO_{\bsalpha})/G^n$ is 
symplectomorphic to $\scN_{\bsalpha} (\RS)$.
\end{proposition}

On the other hand, $\scN^G(\RS)$ admits a structure of 
quasi-Hamiltonian $G^n$-space.
We briefly recall the notion of quasi-Hamiltonian spaces
introduced by Alekseev, Malkin and Meinrenken 
\cite{Alekseev-Malkin-Meinrenken_LGVMM}.

Given a compact connected Lie group $K$ with an invariant inner product
$\langle \, , \, \rangle$ on the Lie algebra $\frakk$, 
let $\theta$ (resp. $\overline{\theta}$) be the left-invariant 
(resp. right-invariant) Maurer-Cartan form,
and let
\[
  \chi = \frac 1{12} \langle \theta, [\theta, \theta] \rangle
  = \frac 1{12} \langle \overline{\theta}, 
  [\overline{\theta}, \overline{\theta}] \rangle
\]
be the canonical bi-invariant 3-form on $K$.

\begin{definition}[{Alekseev, Malkin, and Meinrenken 
\cite[Definition 2.2]{Alekseev-Malkin-Meinrenken_LGVMM}}]
  A {\it quasi-Hamiltonian $K$-space} $M = (M, \omega, \mu)$ is 
  a $K$-manifold $M$ equipped with a $K$-invariant 2-form $\omega$
  and $K$-equivariant map $\mu : M \to K$ such that
  \begin{enumerate}
    \item $d \omega = - \mu^* \chi$,
    \item $\iota(v_{\xi}) \omega = 
          (1/2) \mu^*(\theta + \overline{\theta})$ 
          for each $\xi \in \frakk$, where $v_{\xi}$ is 
          the vector field on $M$ given by the infinitesimal 
          action of $\xi$, and
    \item $\ker \omega_x = \{ v_{\xi}(x) \, | \, \xi \in 
           \ker( \Ad_{\mu(x)} + 1) \}$ for each $x \in M$.
  \end{enumerate}
We call $\mu: M \to K$ the {\it $K$-valued moment map},
or simply the {\it moment map}.
\end{definition}

\begin{example}[The double {\cite[Remark 3.2]{Alekseev-Malkin-Meinrenken_LGVMM}}]
Let $D(G) = G \times G$, and define a $G^2$-action on $D(G)$ by
\[
  (\sigma_1, \sigma_2) \cdot (a, b) := 
  (\sigma_1 a \sigma_2^{-1}, \Ad_{\sigma_2} b)
\]
for $(a,b) \in D(G)$ and $(\sigma_1, \sigma_2) \in G^2$.
Then $D(G)$ is a quasi-Hamiltonian $G^2$-space with the 2-form
\[
  \omega_D = \frac 12 \langle \Ad_b a^* \theta, a^* \theta \rangle
  + \frac 12 \langle a^* \theta, b^* (\theta + \overline{\theta}) \rangle
\]
and the moment map
\[
  \mu = (\mu_1, \mu_2) : D(G) \longrightarrow G^2,
  \quad (a,b) \longmapsto (\Ad_a b, b^{-1}).
\]
\end{example}

\begin{theorem}[{Fusion product
\cite[Theorem 6.1]{Alekseev-Malkin-Meinrenken_LGVMM}}]
Let $(M, \omega, \mu = (\mu_1, \mu_2, \mu_3))$ be a quasi-Hamiltonian 
$K \times K \times H$-space, and consider the diagonal embedding 
$K \times H \hookrightarrow K \times K \times H$, $(k,h) \mapsto (k,k,h)$.
Then $M$ is a quasi-Hamiltonian $K \times H$-space with the 2-form
\[
  \widetilde{\omega} = \omega + 
  \frac 12 \langle \mu_1^* \theta, \mu_2^* \overline{\theta} \rangle
\]
and the moment map
\[
  \widetilde{\mu} = (\mu_1 \cdot \mu_2, \mu_3) 
  : M \longrightarrow K \times H.
\]
\end{theorem}

The product $M_1 \times M_2$ of quasi-Hamiltonian $K \times H_j$-spaces 
$M_j$ ($j =1,2$) is a quasi-Hamiltonian 
$K \times H_1 \times K \times H_2$-space.
The fusion product $M_1 \circledast M_2$ is a 
quasi-Hamiltonian $K \times H_1 \times H_2$-space 
obtained from $M_1 \times M_2$ by fusing $K$-factors.
Note that the fusion product is associative:
\[
  (M_1 \circledast M_2) \circledast M_3 = 
  M_1 \circledast (M_2 \circledast M_3).
\]

We consider $n-1$ copies of double
$D_i = (D(G), \omega_{D_i}, \mu_i)$ ($i = 2, \dots, n$)
with moment map
\[
  \mu_i = (\mu_{i,1}, \mu_{i,2}) \colon D(G)  \longrightarrow
  G^2, \quad 
  (a_i, b_i) \longmapsto (\Ad_{a_i} b_i, b_i^{-1}).
\]
Then the fusion product 
$D(G)^{\circledast (n-1)} = D_2 \circledast \dots \circledast D_n$
given by fusing first $G$-factors is isomorphic to $\scN^G(\RS)$
as a $G^n$-manifold, and hence it defines a structure of 
quasi-Hamiltonian $G^n$-space on $\scN^G(\RS)$.
Since
\begin{align} \label{eq:b1}
  b_1^{-1} = (\Ad_{a_2} b_2) \dots (\Ad_{a_n} b_n)
           = \mu_{2,1} \cdot \mu_{3,1} \cdots \mu_{n,1}
\end{align}
is a component of the moment map on 
$D(G)^{\circledast (n-1)} \cong \scN^G(\RS) $, 
we have the following.

\begin{theorem}[{\cite[Section 9]{Alekseev-Malkin-Meinrenken_LGVMM}}] \label{th:Na_red}
There exists a structure of quasi-Hamiltonian $G^n$-space
on $\scN^G (\RS)$ such that 
\[
  \mu^G : \scN^G (\RS) \longrightarrow G^n,
  \quad (\bsa,\bsb) \longmapsto \bsb^{-1} 
  = (b_1^{-1}, \dots, b_n^{-1})
\]
is the moment map.
The quasi-Hamiltonian reduction
$\bigl(\mu^{G}\bigr)^{-1} (\scC_{\bsalpha})/G^n$
is symplectomorphic to $\scN_{\bsalpha} (\RS)$.
\end{theorem}

\begin{remark}
Treloar \cite{MR1880958} also shows this fact,
and describes the Goldman system as bending Hamiltonians 
on the moduli space of $n$-gons in $S^3 \cong SU(2)$.
\end{remark}

Set $\mu_{\le i} = \mu_{2,1} \cdot \mu_{3,1} \dots \mu_{i,1}$
for simplicity.
Then the 2-form $\omega_{\scN^G(\RS)}$ on $\scN^G(\RS)$ is
given by
\begin{align}
  \omega_{\scN^G(\RS)}
  &= \sum_{i=2}^n \omega_{D_i} + \frac 12 \sum_{i=3}^{n}
     \langle (\mu_{\le i-1})^* \theta , (\mu_{i,1})^* \overline{\theta}
     \rangle  \notag \\
  &= \sum_{i=2}^n \omega_{D_i} + \frac 12 \sum_{i=3}^{n}
     \langle \Ad_{\mu_{i,1}^{-1}} (\mu_{\le i-1})^* \theta ,
     \Ad_{\mu_{i,1}^{-1}} (\mu_{i,1})^* \overline{\theta}
     \rangle  \notag \\
  &= \sum_{i=2}^n \omega_{D_i} + \frac 12 \sum_{i=3}^{n}
     \langle (\mu_{\le i})^* \theta , (\mu_{i,1})^* \theta
     \rangle \notag \\
  &= \sum_{i=2}^n \left( \omega_{D_i} + \frac 12
     \langle (\mu_{\le i})^* \theta , (\mu_{i,1})^* \theta
     \rangle \right).  \label{eq:qHam_2-form}
\end{align}
Here, we have used
\begin{align*}
 \Ad_{\mu_{i,1}^{-1}} \left[ (\mu_{\le i-1})^* \theta \right]
  &= (\mu_{\le i})^* \theta - (\mu_{i,1})^* \theta, \\
 \Ad_{\mu_{i,1}^{-1}} (\mu_{i,1})^* \overline{\theta}
  &= (\mu_{i,1})^* \theta, \\
 \la (\mu_{i,1})^* \theta, (\mu_{i,1})^* \theta \ra &= 0,
\end{align*}
which follow from
\begin{align}
 g^{-1} (h^{-1} d h) g
  &= (hg)^{-1} d(hg) - g^{-1} d g, \\
 g^{-1}((dg)g^{-1})g &= g^{-1} d g,
\end{align}
and the fact that pairing $\la -, - \ra$ is symmetric
and $\theta$ is a one-form.

\section{Walls and quasi-Hamiltonian reductions} 
\label{sc:wall_extended_moduli}

Recall that walls in the space of parabolic weights are given by
\[
  H_{I, k} = \biggl\{ \bsalpha \in [0, 1/2)^n \, \biggm| \,
  \sum_{j \in J} \alpha_j - \sum_{i \in I} \alpha_i = k \biggr\}
\]
for $I \subset \{1, \dots, n \}$,
$J = \{1, \dots, n \} \setminus I$, and $k \in \bZ$.
We define $\bsepsilon = (\epsilon_1, \ldots, \epsilon_n)$ as
\begin{align} \label{eq:epsilon}
 \epsilon_i =
\begin{cases}
 1 & i \in J, \\
 -1 & i \in I,
\end{cases}
\end{align}
so that $\sum_{i=1}^n \epsilon_i \alpha_i = k$.

\begin{lemma} \label{lm:wall1}
A parabolic weight $\alpha \in [0,1)^n$ lies on a wall
if and only if $\cC_{\bsalpha}$ contains $\bsg = (g_1, \dots, g_n)$ 
such that $g_1, \dots, g_n$ lie on a common maximal torus 
and satisfy $g_1 \dots g_n = 1$.
\end{lemma}

\begin{proof}
If $\cC_{\bsalpha}$ contains $\bsg = (g_1, \dots, g_n)$ 
such that $g_1, \ldots, g_n$ lie on a common maximal torus 
and satisfy $g_1 \cdots g_n = 1$,
then one can simultaneously diagonalize $g_1, \ldots, g_n$
so that
$g_i = \exp(\epsilon_i \alpha_i x_0)$
for some $\bsepsilon =(\epsilon_1,\ldots,\epsilon_n) \in \{ \pm 1\}^n$.
Then
$
 g_1 \cdots g_n
  = \exp[(\epsilon_1 \alpha_1+\cdots+\epsilon_n \alpha_n) x_0]
  = 1
$
implies
$\epsilon_1 \alpha_1+\cdots+\epsilon_n \alpha_n \in \bZ$,
so that $\bsalpha$ is on the wall
defined by $\bsepsilon$.

Conversely, if $\bsalpha$ satisfies
$\epsilon_1 \alpha_1+\cdots+\epsilon_n \alpha_n \in \bZ$
for some $\bsepsilon \in \{ \pm 1 \}^n$,
then $(g_i = \exp(\epsilon_i \alpha_i x_0))_{i=1}^n$
gives an element of $\cC_{\bsalpha}$
contained in the same maximal torus
satisfying $g_1 \cdots g_n = 1$.
\end{proof}

Since $\cN_\bsalpha(\RS)$
is described as the quasi-Hamiltonian reduction
$\lb \mu^G \rb^{-1}(\cC_\bsalpha)/G^n$
by \pref{th:Na_red},
there are two ways
for $\cN_\bsalpha(\RS)$ to be singular.
One way is for $\mu^G$ to have a critical point.

\begin{proposition} \label{pr:crit_mu}
The critical point set of $\mu^G$ consists of 
$(\bsa, \bsb) \in \scN^G(\RS)$ such that
$b_1, a_2b_2a_2^{-1}, \dots, a_n b_n a_n^{-1}$
lie on a common maximal torus.
\end{proposition}

\begin{proof}
Suppose that $\mu^G(\bsa, \bsb) = \bsb^{-1} \in \scC_{\bsalpha}$.
Under the identifications 
$T_{(\bsa, \bsb)} \scN^G (\RS) \cong T_{(\bsa, \bsb)}G^{2(n-1)}
\cong \frakg^{2(n-1)}$  and
$T_{\bsb^{-1}}G^n \cong \frakg^{n}$ by right translations,
$d \mu^G_{(\bsa, \bsb)} : \frakg^{2(n-2)} \to \frakg^{n}$ 
is given by
\[
  d \mu^G_{(\bsa, \bsb)} (\xi_2, \dots, \xi_n, \eta_2, \dots, \eta_n)
  = (-\Ad_{b_1^{-1}}\eta_1, \dots, -\Ad_{b_n^{-1}}\eta_n)
\]
with
\[
  -\Ad_{b_1^{-1}}\eta_1 = \sum_{i=2}^n
  \Ad_{(a_2b_2a_2^{-1}) \dots  (a_{i-1}b_{i-1}a_{i-1}^{-1})}
  \bigl(\xi_i - \Ad_{a_ib_ia_i^{-1}} \xi_i - \Ad_{a_i} \eta_i \bigr).
\]
Hence $(\bsxi, \bseta) \in \ker d \mu^G_{(\bsa, \bsb)}$ if and only if
$\bseta=0$ and
\begin{equation}
  \sum_{i=2}^n
  \Ad_{(a_2b_2a_2^{-1}) \dots  (a_{i-1}b_{i-1}a_{i-1}^{-1})}
  \bigl(\xi_i - \Ad_{a_ib_ia_i^{-1}} \xi_i \bigr) = 0.
 \label{eq:ker_dmu}
\end{equation}
Since $b_i \in \scC_{\alpha_i}$, there exists $g_i \in G$ such that
$b_i = g_i e^{\alpha_i x_0} g_i^{-1}$.
Setting
\begin{align*}
  h_i &= (a_2b_2a_2^{-1}) \dots  (a_{i-1}b_{i-1}a_{i-1}^{-1}) a_i g_i,\\
  \xi'_i &= \Ad_{g_i^{-1} a_i^{-1}} \xi_i,
\end{align*}
the equation \eqref{eq:ker_dmu} is written as
\begin{align} \label{eq:ker2}
  \sum_{i=2}^n \Ad_{h_i} \bigl( \xi'_i - 
  \Ad_{\exp(\alpha_i x_0)} \xi'_i \bigr) = 0.
\end{align}
Since
\[
  \xi' - \Ad_{\exp(\alpha_i x_0)} \xi'
  = 
  \begin{pmatrix} 
    0 & (1-e^{4\pi\sqrt{-1} \alpha_i}) \xi'_{12} \\
    (1-e^{-4\pi\sqrt{-1} \alpha_i}) \xi'_{21} & 0
  \end{pmatrix}
\]
for $\xi' = (\xi'_{ij}) \in \frakg$, 
the dimension of the image of the map
\begin{align*}
 \frakg^{n-1} \to \frakg, \quad
 (\xi_2',\ldots,\xi_n') \mapsto
  \sum_{i=2}^n \Ad_{h_i} \bigl( \xi'_i - 
  \Ad_{\exp(\alpha_i x_0)} \xi'_i \bigr)
\end{align*}
is at least two,
and is exactly two
if and only if there exists some $g \in G$
such that $g h_2, \dots, g h_n$ are diagonal matrices.
Note that
\[
  (g h_i) e^{\alpha_i x_0} (g h_i)^{-1}
  = g \bigl( (a_2 b_2 a_2^{-1}) \dots (a_{i-1} b_{i-1} a_{i-1}^{-1}) \bigr)
  (a_i b_i a_i^{-1})
  \bigl( (a_2 b_2 a_2^{-1}) \dots (a_{i-1} b_{i-1} a_{i-1}^{-1}) \bigr)^{-1}
   g^{-1}.
\]
If $g h_2$ is a diagonal matrix,
then so is $g a_2 b_2 a_2^{-1} g^{-1}$.
If $h_3$ is also a diagonal matrix,
then so is
$g (a_2 b_2 a_2^{-1}) (a_3 b_3 a_3^{-1}) (a_2 b_2 a_2^{-1})^{-1} g^{-1}$,
and hence $g a_3 b_3 a_3^{-1} g^{-1}$ is also a diagonal matrix.
By continuing the same discussion,
one shows that if $g h_2, \ldots, g h_n$ are diagonal matrices,
then so are $g a_i b_i a_i^{-1} g^{-1}$ for $n=2, \ldots, n$ are.
Then \eqref{eq:b1} implies that $g b_1 g^{-1}$ is also a diagonal matrix.
This means that
$b_1, a_2b_2a_2^{-1}, \dots, a_nb_na_n^{-1}$ are
in the same maximal torus,
and \pref{pr:crit_mu} is proved.
\end{proof}

The other way for $\cN_\alpha(\RS)$
to be singular
is for the $G^n$-action
on the level set
$\scN^G(\RS ; \bsalpha) = \bigl(\mu^G\bigr)^{-1}(\scC_{\bsalpha})$
to have more stabilizer
than the generic orbit.
Note that the generic stabilizer is given by
$\{ \pm \bsone \} = \{ \pm (1, \dots, 1) \}  \subset G^n$.

\begin{proposition} \label{pr:non-free}
The non-free locus of the $G^n / \{ \pm \bsone \}$-action
on $\scN^G(\RS ; \bsalpha)$ consists of
$(\bsa, \bsb) \in \scN^G(\RS)$ such that
$b_1, a_2b_2a_2^{-1}, \dots, a_n b_n a_n^{-1}$
lie on a common maximal torus.
\end{proposition}

\begin{proof}
Suppose that $\bssigma = (\sigma_1, \dots, \sigma_n) \in G^n$ fixes 
$(\bsa,\bsb) \in \scN^G(\RS; \bsalpha)$, i.e.,
\begin{align}
  \sigma_1 a_l \sigma_l^{-1} &= a_l, \qquad l=2, \dots, n, 
  \label{eq:stab1} \\
  \sigma_i b_i \sigma_i^{-1} &= b_i, \qquad i=1, \dots, n.
  \label{eq:stab2}
\end{align}
The condition \eqref{eq:stab1} is written as 
$\sigma_l = a_l^{-1} \sigma_1 a_l$,
which means that $\sigma_1, \dots, \sigma_n$ are in the same
conjugacy class.
By the $G^n$-action
\[
  a_l \mapsto g_1a_lg_l^{-1}, \quad
  b_i \mapsto g_ib_ig_i^{-1}, \quad
  \sigma_i \mapsto g_i\sigma_i g_i^{-1}, 
\]
we may assume that 
$b_i = e^{\alpha_i x_0}$
are diagonal matrices for $i=1, \dots, n$.
Then \eqref{eq:stab2} implies that $\sigma_i$ is a diagonal matrix if
$b_i \ne 1$.
We may assume that $\sigma_i$ is diagonal also in the case $b_i=1$
by the $G^n$-action. 
Since $\sigma_1,\ldots,\sigma_n$ are diagonal matrices
in the same conjugacy class,
one has $\sigma_i = \sigma_1^{\epsilon_i}$
for some diagonal matrix $\sigma_1$ and
$\epsilon_i \in \{ \pm 1 \}$.
Now we assume that $\bssigma \ne \pm \bsone$.
This implies $\sigma_i \ne \pm 1$ for all $i=1, \dots , n$,
since $(\pm 1)^{-1} = \pm 1$.
From \eqref{eq:stab1}, $a_l$ has the form
\[
  a_l = \begin{pmatrix} 0 & 1 \\ 1 & 0 \end{pmatrix}^{(1-\epsilon_l)/2}
  \begin{pmatrix}
    e^{2 \pi \sqrt{-1} \tau_l} & 0 \\
    0 & e^{-2 \pi \sqrt{-1} \tau_l}
  \end{pmatrix},
\]
and hence 
\[
  a_l b_l a_l^{-1} = \begin{pmatrix}
    e^{2 \pi \sqrt{-1} \epsilon_l \alpha_l} & 0 \\
    0 & e^{-2 \pi \sqrt{-1} \epsilon_l \alpha_l}
  \end{pmatrix}.
\]
The condition $b_1 (a_2 b_2 a_2^{-1}) \dots (a_n b_n a_n^{-1})=1$
implies that $\sum_i \epsilon_i \alpha_i = k \in \mathbb{Z}$,
which means that $\bsalpha \in H_{I,k}$ for 
$I = \{ i \, | \, \epsilon_i = 1 \}$.

Conversely, if $\bsalpha \in H_{I,k}$, then the above argument 
shows that there exists a set of diagonal matrices 
$(\bsa, \bsb) \in \scN^G(\RS ; \bsalpha)$ 
which has a non-trivial stabilizer.
\end{proof}

The proof of \pref{pr:non-free} shows that
any element of the stabilizer of 
$(\bsa, \bsb) \in \scN^G(\RS; \bsalpha)$ has the form
\[
  \bssigma 
  = ( \sigma_1, \dots, \sigma_n) 
  = (\sigma_1, a_2^{-1} \sigma_1 a_2, \dots, a_n^{-1}\sigma_1 a_n).
\]
Note that $b_i = \pm 1$ if and only if $\alpha_i \in \{ 0, 1/2 \}$.
If $b_i \ne \pm 1$ for some $i$, then \eqref{eq:stab2} implies that 
$\sigma_i = a_i^{-1}\sigma_1 a_i$ must be in a maximal torus,
and hence the stabilizer is isomorphic to $T$.
On the other hand, if $b_i = \pm 1$ for all $i$, then the stabilizer is 
isomorphic to $G$, since $\sigma_1$ can be arbitrary.
When $\bsalpha \in \{ 0, 1/2 \}^n$,
then $\bsb \in \{ \pm 1 \}^n$ carries no degree of freedom
and the $\bsa$-projection
induces an isomorphism of $\cN^G(\RS; \bsalpha)$
with $G^{n-1}$.
The $G^n$ action on $G^{n-1}$ indeed has a stabilizer isomorphic to $G$,
and the quotient $\cN_\bsalpha(\RS) = \cN^G(\RS; \bsalpha)/G^n$
consists of one point.

Propositions \ref{pr:crit_mu} and \ref{pr:non-free} shows that
if $\bsalpha$ lies on some $H_{I,k}$,
then the singular locus of $\cN_\bsalpha(\RS)$
is given by $[g_1,\ldots,g_n] \in \cC_\bsalpha / G$
such that $g_1, \ldots, g_n$ lie on a common maximal torus.
Then one can diagonalize $g_1, \ldots, g_n$ simultaneously,
so that $g_i = \exp(\epsilon_i \alpha_i x_0)$
where $\bsepsilon$ are given in \eqref{eq:epsilon}.
If $\bsalpha$ lies on $k$ walls,
then $\cN_\bsalpha(\RS)$ has $k$ isolated singularities,
each of which is given by
$
 [\exp(\epsilon_1 \alpha_1 x_0), \ldots, \exp(\epsilon_n \alpha_n x_0)].
$

\begin{corollary} \label{cor:stratification}
Suppose that $\bsalpha$ is a weight lying on some $H_{I,k}$.
Let $(\bsa, \bsb) \in \scN^G(\RS; \bsalpha)$ be a critical point of $\mu^G$,
and $\bsg \in \scN_{\bsalpha}(\RS)$ be the corresponding singular point.
Then there exists an open neighborhood $U \subset \scN_{\bsalpha}(\RS)$ 
of $\bsg$ such that $\scN^G(\RS; \bsalpha)$ is locally homeomorphic to
$
 (\frakg^n/ \frakt) \times
  \big( (U \times T) \big/ (\{\bsg \} \times T) \big).
$
In particular, $\scN^G(\RS; \bsalpha)$ admits a 
$G^n$-invariant Whitney stratification.
\end{corollary}

Here $(U \times T) \big/ (\{\bsg \} \times T)$
is the topological space
obtained from $U \times T$
by contracting the subset $\{\bsg \} \times T \subset U \times T$
to a point, and
$\frakg^n / \frakt$ is the quotient vector space.

Propositions \ref{pr:crit_mu} and \ref{pr:non-free} implies the following.

\begin{corollary} \label{cor:diffeo_Na}
If $\bsalpha$ and $\bsalpha'$ are in the same chamber, then
$\scN_{\bsalpha}(\RS)$ is diffeomorphic to $\scN_{\bsalpha'}(\RS)$.
\end{corollary}

Let 
\[
  \mu^T = \mu^G|_{\scN^T(\RS)} : \scN^T(\RS) \longrightarrow T^n,
  \quad (\bsa, \bsb) \longmapsto \bsb^{-1}
\]
be the restriction of the group-valued moment map.
Then
\[
  (\mu^G)^{-1}(\scC_{\bsalpha}) \cap \scN^T(\RS)
  = (\mu^T)^{-1}(e^{-\alpha_1 x_0}, \dots, e^{- \alpha_n x_0}).
\]

\begin{corollary}
If $\bsalpha \nin \{0, 1/2\}^n$, then the diffeomorphism
$(\mu^G)^{-1}(\scC_{\bsalpha})/G^n \cong \scN_{\bsalpha}(\RS)$ induces
\[
  (\mu^T)^{-1}(e^{-\alpha_1 x_0}, \dots, e^{- \alpha_n x_0})/T^n
  \cong \scN_{\bsalpha}(\RS).
\]
\end{corollary}

\section{Gluing and Goldman systems} \label{sc:gluing}

In this section, we see the Goldman's functions 
via gluing of Riemann surfaces, following the idea of
Hurtubise and Jeffrey \cite{Hurtubise-Jeffrey_RWFFPB}.

Fix a simple closed curve $C$ in $\RS$ and consider a decomposition 
$\RS = \RS^+ \cup_C \RS^-$ into two surfaces
by cutting $\RS$ along $C$.
We may assume that the boundary components of $\RS^+$ (resp. $\RS^-$)
are $B_1^+ = B_1, \dots , B_{m+1}^+ = B_{m+1}$, and  $B_{m+2}^+ = C$ 
(resp. $B_1^- = C, B_2^- = B_{m+2}, \dots, B_{n-m}^- = B_{n}$).
Then $\scN^G (\RS^+)$ 
(resp. $\scN^G (\RS^-)$) has the action of 
$G^{m+2} = G_1^+ \times \dots \times G_{m+2}^+$
(resp. $G^{n-m} = G_{1}^- \times \dots \times G_{n-m}^-$)
corresponding to the boundary components.
We write the moment maps $\mu^G_{\pm}$ on $\scN^G(\RS^{\pm})$ as
\begin{align*}
  \mu^G_+ &= (\mu^G_{B_1^+}, \dots, \mu^G_{B_{m+2}^+})
    =(\mu^+_{\le m+2}, \mu^+_{2,2}, \dots, \mu^+_{m+2,2}), \\
  \mu^G_- &= (\mu^G_{B_1^-}, \dots, \mu^G_{B_{n-m}^-})
    =(\mu^-_{\le n-m}, \mu^-_{2,2}, \dots, \mu^-_{n-m,2}).
\end{align*}
For the diagonal subgroup $G_C \subset G_{m+2}^+ \times G_1^-$, 
The moment map of the $G_C$-action on the fusion product
$\scN^G ( \RS^+ \amalg \RS^-) := 
\scN^G(\RS^+) \circledast \scN^G(\RS^-)$
is given by
\[
  \nu_C^{G} = \mu^G_{B^+_{m+2}} \cdot \mu^G_{B^-_1} :
  \scN^G(\RS^+ \amalg \RS^-)
  \longrightarrow G, \quad
  \bigl( (\bsa^+,\bsb^+), (\bsa^-, \bsb^-) \bigr)
  \longmapsto
  (b^-_1 b^+_{m+2})^{-1}.
\]
We define the ``gluing map'' 
$\pi^G_C : \bigl(\nu^G_C \bigr)^{-1}(1) \to \scN^G(\RS)$
by
\begin{multline*}
  \pi_C \bigl( (\bsa^+,\bsb^+), (\bsa^-, \bsb^-) \bigr) \\
  = (a_2^+, \dots, a_{m+1}^+, a_{m+2}^+ a_2^-, \dots, a_{m+2}^+ a_{n-m}^-;
     b_1^+, \dots, b_{m+1}^+, b^-_2, \dots, b_{n-m}^-).
\end{multline*}
(See Figure \ref{fig:tree}.)

\begin{figure}[h]
 \begin{center}
  \input{fig_tree.pst}
 \end{center}
 \caption{The dual graph of $\RS^+ \amalg \RS^-$.}
 \label{fig:tree}
\end{figure}

Then we have

\begin{proposition}
The map
$\pi_C : \bigl(\nu^G_C \bigr)^{-1}(1) \to \scN^G(\RS)$
induces an isomorphism
\[
  \bigl(\nu^G_C \bigr)^{-1}(1)/G_C \cong \scN^G(\RS)
\]
of quasi-Hamiltonian $G^n$-space.
\end{proposition}

\begin{proof}
It is easy to see that $\pi_C$ is well-defined and surjective.
To see that the induced map is injective, suppose that
\[
  \pi_C \bigl( (\bsa^+,\bsb^+), (\bsa^-, \bsb^-) \bigr) =
  \pi_C \bigl( (\bsc^+,\bsd^+), (\bsc^-, \bsd^-) \bigr)
\]
for $\bigl( (\bsa^+,\bsb^+), (\bsa^-, \bsb^-) \bigr),
\bigl( (\bsc^+,\bsd^+), (\bsc^-, \bsd^-) \bigr) \in 
 \bigl(\nu^G_C \bigr)^{-1}(1)$.
Then we have
\begin{align}
  a^+_i &= c_i^+, \quad i = 2, \dots, m+1, \notag \\
  a_{m+2}^+a_j^- &= c_{m+2}^+c_j^-, \quad j = 2, \dots, n-m, \label{eq:transf_ac}\\
  b_i^+ &= d_i^+, \quad i = 1, \dots, m+1, \notag \\
  b_j^- &= d_j^-, \quad j = 2, \dots, n-m. \notag
\end{align}
Note that $b_{m+2}^+, b_1^-, d_{m+2}^+, d_1^-$ are determined by
\begin{align*}
  b_1^+ \bigl(a_2^+ b_2^+ (a_2^+)^{-1} \bigr) \dots 
  \bigl(a_{m+2}^+ b_{m+2}^+ (a_{m+2}^+)^{-1} \bigr) &= 1, \\
  b_1^- \bigl(a_2^- b_2^- (a_2^-)^{-1} \bigr) \dots 
  \bigl(a_{n-m}^+ b_{n-m}^+ (a_{n-m}^+)^{-1} \bigr) &= 1, \\
  d_1^+ \bigl(c_2^+ d_2^+ (c_2^+)^{-1} \bigr) \dots 
  \bigl(c_{m+2}^+ d_{m+2}^+ (c_{m+2}^+)^{-1} \bigr) &= 1, \\
  d_1^- \bigl(c_2^- d_2^- (c_2^-)^{-1} \bigr) \dots 
  \bigl(c_{n-m}^+ d_{n-m}^+ (c_{n-m}^+)^{-1} \bigr) &= 1.
\end{align*}
Setting $\sigma = (c_{m+2}^+)^{-1} a_{m+2}^+ \in G = G_C$, 
the condition \eqref{eq:transf_ac} is written as
\[
  c_j^- = \sigma a_j^-, \quad j = 2, \dots, n-m.
\]
This implies that
$\bigl( (\bsc^+,\bsd^+), (\bsc^-, \bsd^-) \bigr)
= \sigma \cdot \bigl( (\bsa^+,\bsb^+), (\bsa^-, \bsb^-) \bigr)$.
Hence the induced map $(\nu^G_C)^{-1}(1)/G \to \scN^G(\RS)$ is injective. 

It remains to check that 
$\iota^* \omega_{\scN^G(\RS^+ \amalg \RS^-)}
= \pi_C^* \omega_{\scN^G(\RS)}$, 
where $\iota : (\nu_C^G)^{-1}(1) \hookrightarrow \scN^G( \RS^+ \amalg \RS^-)$
is the inclusion.
From \eqref{eq:qHam_2-form}, the 2-form $\omega_{\scN^G(\RS^+ \amalg \RS^-)}$ 
on $\scN^G(\RS^+ \amalg \RS^-)$ is given by
\begin{align*}
  \omega_{\scN^G(\RS^+ \amalg \RS^-)} 
  &= \omega_{\scN^G(\RS^+)} + \omega_{\scN^G(\RS^-)}
  + \frac 12 \langle (\mu^+_{m+2,2})^*\theta, 
     (\mu_{\le n-m}^-)^* \overline{\theta} \rangle \\
  &= \sum_{i=2}^{m+2} \biggl( \omega_{D_i^+} + \frac 12 
     \langle (\mu_{\le i}^+)^* \theta , (\mu_{i,1}^+)^* \theta
     \rangle \biggr) \\
  &\quad + \sum_{j=2}^{n-m} \biggl( \omega_{D_i^-} + \frac 12 
     \langle (\mu_{\le j}^-)^* \theta , (\mu_{j,1}^-)^* \theta
     \rangle \biggr) \\
  &\quad + \frac 12 \langle (\mu_{m+2,2}^+)^*\theta, 
     (\mu_{\le n-m}^-)^* \overline{\theta} \rangle.
\end{align*}
Since $(\mu_{m+2,2}^+)^{-1} = b_{m+2}^+ = (b_1^-)^{-1} = \mu_{\le n-m}^-$ 
and $\mu_{m+1,1}^+ = \Ad_{a_{m+2}^+}\mu_{\le n-m}^-$ 
on $(\nu^G_C)^{-1}(1)$, we have
\[
  \iota^* \omega_{D_{m+2}^+} = 
  \Bigl\langle (a_{m+2}^+)^*, 
  (\mu_{\le n-m}^-)^*(\theta + \overline{\theta})
  - \Ad_{\mu_{\le n-m}^-} (a_{m+2}^+)^* \theta \Bigr\rangle
\]
and
\begin{align*}
  \iota^* \langle (\mu_{\le m+2}^+)^* \theta , 
     (\mu_{m+2,1}^+)^* \theta \rangle
  &= \Bigl\langle (\mu_{\le m+1}^+)^* \theta , 
     (\Ad_{a_{m+2}^+} \mu_{\le n-m}^-)^* \overline{\theta} \Bigr\rangle \\
  \iota^* \langle (\mu_{m+2,2}^+)^*\theta, 
     (\mu_{\le n-m}^-)^* \overline{\theta} \rangle &= 0.
\end{align*}
Then the restriction $\iota^* \omega_{\scN^G(\RS^+ \amalg \RS^-)}$ 
is given by
\begin{align*}
  \iota^* \omega_{\scN^G(\RS^+ \amalg \RS^-)} 
  &= \sum_{i=2}^{m+1} \biggl( \omega_{D_i^+} + \frac 12 
     \langle (\mu_{\le i}^+)^* \theta , (\mu_{i,1}^+)^* \theta
     \rangle \biggr) \\
  &\quad + \sum_{j=2}^{n-m} \biggl( \omega_{D_i^-} + \frac 12 
     \langle (\mu_{\le j}^-)^* \theta , (\mu_{j,1}^-)^* \theta
     \rangle \biggr) \\
  &\quad + \frac 12 \Bigl\langle (\mu_{\le m+1}^+)^*\theta, 
     (\Ad_{a_{m+2}^+} \mu_{\le n-m}^-)^* \overline{\theta} \Bigr\rangle \\
  &\quad + \frac 12 \Bigl\langle (a_{m+2}^+)^*\theta, 
     (\mu_{\le n-m}^-)^* (\theta + \overline{\theta}) -
     \Ad_{\mu_{\le n-m}^-} (a_{m+2}^+)^* \theta \Bigr\rangle.
\end{align*}
On the other hand, the pull-back of $\omega_{\scN^G(\RS)}$ is given by
\begin{align*}
  \pi_C^*\omega_{\scN^G(\RS)}
  &=  \pi_C^* \sum_{i=2}^{m+1}
     \biggl( \omega_{D_i} + \frac 12 
     \langle (\mu_{\le i})^* \theta , (\mu_{i,1})^* \theta
     \rangle \biggr) \\
  &\quad + \pi_C^* \sum_{j=2}^{n-m}
     \biggl( \omega_{D_{m+j}} + \frac 12 
     \langle (\mu_{\le {m+j}})^* \theta , (\mu_{m+j,1})^* \theta
     \rangle \biggr)
\end{align*}
with
\begin{equation}
  \pi_C^* \biggl( \omega_{D_i} + \frac 12 
     \langle (\mu_{\le i})^* \theta , (\mu_{i,1})^* \theta
     \rangle \biggr)
  = \omega_{D_i^+} + \frac 12 
     \langle (\mu_{\le i}^+)^* \theta , (\mu_{i,1}^+)^* \theta
     \rangle 
  \label{eq:q-ham_red_1}
\end{equation}
for $i = 2, \dots, m+1$.
By using 
\[
  \pi_C^* a_{m+j}^* \theta = 
  (a_{m+2}^+ a_j^-)^* \theta =
  \Ad_{(a_j^-)^{-1}}
  (a_{m+2}^+)^*\theta + (a_j^-)^* \theta
\]
for $j=2, \dots, n-m$ and formulae
\begin{align}
  (\Ad_a b)^* \theta 
  &= \Ad_{ab^{-1}} a^*\theta + \Ad_a b^* \theta 
     - a^* \overline{\theta}, 
  \label{eq:formula_MC1} \\
  (\Ad_a b)^* \overline{\theta}
  &= a^* \overline{\theta} + \Ad_a b^* \overline{\theta}
     - \Ad_{ab} a^* \theta,
  \label{eq:formula_MC2}
\end{align}
we have
\begin{align}
  \pi_C^* \omega_{D_{m+j}}
  &= \omega_{D_j^-} + \frac 12
     \Bigl\langle (a_{m+2}^+)^* \theta, 
      (\mu_{j,1}^-)^* (\theta + \overline{\theta}) 
      - \Ad_{\mu_{j,1}^-}(a_{m+2}^+)^* \theta \Bigr\rangle.
  \label{eq:q-ham_red_2}
\end{align}
Similarly,
\begin{align}
  \pi_C^* &\langle (\mu_{\le m+j})^* \theta,
  (\mu_{m+j, 1})^* \theta \rangle \notag \notag \\
  &= \Bigl\langle (\mu_{\le m+1}^+ \cdot \Ad_{a_{m+2}^+}\mu_{\le j}^-)^* 
     \theta , (\Ad_{a_{m+2}^+}\mu_{j,1}^-)^* \theta \Bigr\rangle \notag \\
  &= \Bigl\langle \Ad_{(\Ad_{a_{m+2}^+} \mu_{\le j}^-)^{-1}} 
      (\mu_{\le m+1}^+)^*\theta 
      + (\Ad_{a_{m+2}^+}\mu_{\le j}^-)^* \theta, 
      (\Ad_{a_{m+2}^+}\mu_{j,1}^-)^* \theta \Bigr\rangle \notag \\
  &= \Bigl\langle (\mu_{\le m+1}^+)^*\theta,
       \Ad_{\Ad_{a_{m+2}^+} \mu_{\le j}^-} 
       (\Ad_{a_{m+2}^+}\mu_{j,1}^-)^* \theta \Bigr\rangle \notag \\
  & \hspace{5cm} + \Bigl\langle (\Ad_{a_{m+2}^+}\mu_{\le j}^-)^* \theta, 
      (\Ad_{a_{m+2}^+}\mu_{j,1}^-)^* \theta \Bigr\rangle, \notag
\end{align}
and the formulae \eqref{eq:formula_MC1} and \eqref{eq:formula_MC2}
imply 
\begin{align}
  \pi_C^* &\langle (\mu_{\le m+j})^* \theta,
  (\mu_{m+j, 1})^* \theta \rangle \notag \notag \\
  &= \langle (\mu_{\le j}^-)^* \theta, (\mu_{j,1}^-)^* \theta \rangle 
    + \Bigl\langle (a_{m+2}^+)^* \theta, 
       \Ad_{\mu_{j,1}^-} (a_{m+2}^+)^* \theta 
       - (\mu_{j,1}^-)^* (\theta + \overline{\theta}) \Bigr\rangle \notag \\
  &\quad - \Bigl\langle (a_{m+2}^+)^* \theta, 
       \Ad_{\mu_{\le j}^-} (a_{m+2}^+)^*\theta
       - \Ad_{\mu_{\le j-1}^-} (a_{m+2}^+)^* \theta \Bigr\rangle \notag \\
  &\quad + \langle (a_{m+2}^+)^* \theta, 
       (\mu_{\le j}^-)^* (\theta + \overline{\theta}) 
       - (\mu_{\le j-1}^-)^* (\theta + \overline{\theta}) \rangle \notag \\
  &\quad +  \Bigl\langle (\mu_{\le m+1}^+)^*\theta, 
     (\Ad_{a_{m+2}^+} \mu_{\le j}^-)^* \overline{\theta} 
     - (\Ad_{a_{m+2}^+} \mu_{\le j-1}^-)^* \overline{\theta} \Bigr\rangle,
  \label{eq:q-ham_red_3}
\end{align}
where we assume that $\mu_{\le 1}^- = 1$ is a constant map.
Combining \eqref{eq:q-ham_red_1}, \eqref{eq:q-ham_red_2}, and 
\eqref{eq:q-ham_red_3}, we have
$
  \pi_C^*\omega_{\scN^G(\RS)}
  = \iota^* \omega_{\scN^G(\RS^+ \amalg \RS^-)}.
$
\end{proof}

We consider the action of 
$G_{m+2}^+ = G_{m+2}^+ \times \{1\} \subset G_{m+2}^+ \times G_1^-$
with moment map
\[
  \mu^{G}_C = \mu^G_{B^+_{m+2}} : 
\scN^{G}(\RS^+ \amalg \RS^-)\to G, \quad 
  \mu_C^G (\bsa^{\pm},\bsb^{\pm}) 
  = (b_{m+2}^+)^{-1}.
\]
Since $G^{n-3}$ acts on the $b_{m+2}^+$-component by conjugation, 
the function
\[
  \tilde{\vartheta}_C = \cos^{-1}\left( \frac 12 \tr \mu^{G}_C \right)
  : \scN^{G}(\RS^+ \amalg \RS^-) \longrightarrow \bR
\]
descends to  $\scN^G(\RS)$,
and induces a Goldman's function
$\vartheta_C : \scN_{\bsalpha}(\RS) \to \bR$.
Let $\nu_C^T = \nu_C^G|_{\scN^T(\RS^+ \amalg \RS^-)}$ be the 
restriction of the moment map to 
$\scN^T(\RS^+ \amalg \RS^-) = \scN^T(\RS^+) \times \scN^T(\RS^-)$.
Then $(\nu^T_C)^{-1}(1) \subset (\nu^G_C)^{-1}(1)$ is
preserved under the action of the maximal torus
$T^+_{m+2} \times T^-_1 \subset G^+_{m+2} \times G^-_1$.
The Hamiltonian torus action of $\vartheta_C$ is induced from the 
action of $T_{m+2}^+ \times \{1 \} \subset T^+_{m+2} \times T^-_1$ 
on $(\nu^T_C)^{-1}(1)$ (see \cite{Hurtubise-Jeffrey_RWFFPB}).
 
Now we fix a pair-of-pants decomposition 
$\Sigma = \bigcup_{i=1}^{n-2} \RS_i$ given by $n-3$ simple closed curves
$C_1, \dots, C_{n-3}$ with dual graph $\Gamma$,
and let $C_i^+$, $C^-_i$ denote the copies of $C_i$ in 
the disjoint union $\coprod_i \RS_i$.
Then the fusion product 
$\scN^G (\coprod_i \RS_i) :=
 \scN^G (\RS_1) \circledast \dots \circledast \scN^G(\RS_{n-2})$ 
has the actions of diagonal subgroups 
$G^{n-3} = \prod_i G_{C_i}$ in $G^{2(n-3)} = \prod_i G_{C_i^+} \times G_{C_i^-}$
with moment map
$\nu^{G}_{\Gamma} : \scN^{G}(\coprod_i \RS_i) \to G^{n-3}$.
We can define the gluing map
$\pi_{\Gamma} : 
\bigl(\nu^G_{\Gamma} \bigr)^{-1}(\bsone) \to \scN^G(\RS)$
in a similar manner.

\begin{corollary}
The map
$\pi_{\Gamma} : \bigl(\nu^G_{\Gamma} \bigr)^{-1}(\bsone) \to \scN^G(\RS)$
induces an isomorphism
\[
  \bigl(\nu^G_{\Gamma} \bigr)^{-1}(\bsone)/G^{n-3} \cong \scN^G(\RS)
\]
of quasi-Hamiltonian $G^n$-spaces.
The functions $\tilde{\vartheta}_{C_1}, \dots, \tilde{\vartheta}_{C_{n-3}}$ induces
the Goldman system
\[
  \Theta_{\bsalpha, \Gamma} = (\vartheta_{C_1}, \dots, \vartheta_{C_{n-3}}) :
  \scN_{\bsalpha} (\RS) \longrightarrow \bR^{n-3}.
\]
The Hamiltonian torus action of $\Theta_{\bsalpha, \Gamma}$ is given by the action of the
maximal torus 
$\prod_{i=1}^{n-3} T_{C_i^+} \subset \prod_{i=1}^{n-3} (G_{C_i^+} \times \{1\})$
on $(\nu^T_{\Gamma} \bigr)^{-1}(\bsone) \subset \scN^{T}(\coprod_i \RS_i)$.
\end{corollary}

\begin{remark}
The reduction $(\nu^T_{\Gamma})^{-1}(\bsone)/T^{n-3}$ of 
the $T$-extended moduli space $\scN^T(\coprod_i \RS_i)$ 
is not homeomorphic to $\scN^T(\RS)$ on the locus where
holonomies along any components of $\partial \RS_i$ are central for some $i$.
\end{remark}

\section{Isomorphisms of Goldman systems} \label{sc:isom_Goldman}

Fix a generic parabolic weight 
$\bsalpha \in (0, 1/2)^n$ such that $|\bsalpha| < 1$. 
Then $t \bsalpha = (t \alpha_1, \dots, t \alpha_n)$ and $\bsalpha$ 
are in the same chamber for $t \in (0,1]$, and hence 
$\scN_{t \bsalpha}(\RS)$ is diffeomorphic to 
$\scN_{\bsalpha}(\RS)$ for $t \in (0,1]$.
Note that the images 
$\Delta_{\Gamma}(t\bsalpha) = 
\Theta_{t\bsalpha, \Gamma}(\scN_{t \bsalpha}(\RS))$
of the Goldman systems are related by scalings
$\Delta_{\Gamma}(t\bsalpha) = t \Delta_{\Gamma}(\bsalpha)$.
In this section we prove the following theorem.

\begin{theorem} \label{thm:isom1}
Suppose that $\bsalpha$ satisfies the above condition.
Then for each $\Gamma$, there exists a family of symplectomorphism
\[
  \psi_t : (\scN_{\bsalpha}(\RS) , \omega_{\scN_{\bsalpha}}) 
  \to (\scN_{t \bsalpha}(\RS), (1/t) \omega_{\scN_{t \bsalpha}})
\]
such that 
$(1/t) \psi_t^* \Theta_{t \bsalpha, \Gamma} = \Theta_{\bsalpha, \Gamma}$.
Namely, 
\begin{align} \label{eq:fN}
  \frakN(\RS) 
  = \bigcup_{t \in (0,1]} \scN_{t \bsalpha} (\RS)
  \to (0,1]
\end{align}
is trivial as a family of symplectic manifolds equipped with
completely integrable systems.
\end{theorem}

We first consider a decomposition 
$\RS = \RS^+ \cup_C \RS^-$ given by a single simple closed curve
as in Section \ref{sc:gluing}.

\begin{lemma}
For $t \in (0,1]$, there exists a diffeomorphism
$\psi_t : \scN_{\bsalpha}(\RS) \to \scN_{t \bsalpha}(\RS)$
such that $(1/t) \psi_t^* \vartheta_{t\bsalpha, C} = \vartheta_{\bsalpha, C}$.
\end{lemma}

\begin{proof}
Let $\frakC = \bigcup_{t \in (0,1]} \scC_{t \bsalpha} \subset G^n$
be the family of conjugacy classes with projection
$\pi_{\frakC}: \frakC \to (0,1]$.
Then the total space $\frakN (\RS)$
of the family \eqref{eq:fN} is given by
\[
  \frakN(\RS) = (\mu^G)^{-1}( \frakC) / G^n,
\]
where $\mu^G : \scN^G(\RS) \to G^n$ is the moment map.
Since $|\bsalpha|<1$, the family $\frakC$ is trivialized by
\begin{equation}
  \scC_{\bsalpha} \longrightarrow \scC_{t \bsalpha},
  \quad
  \bsc = (c_1, \dots, c_n) \longmapsto
  \bsc^t = ((c_1)^t, \dots, (c_n)^t),
  \label{eq:triv_sing}
\end{equation} 
where $c^t = g e^{t x} g^{-1}$ for 
$c = g e^{x} g^{-1} \in \scC_{\alpha}$ with $x \in A_{\frakt}$.

Let 
\[
  \mu^G_{\partial \RS} 
  = (\mu^G_{B_1^+}, \dots, \mu^G_{B_{m+1}^+}, 
     \mu^G_{B_2^-}, \dots, \mu^G_{B_{n-m}^+})
  : \scN^G(\RS^+ \amalg \RS^-) \longrightarrow G^n
\]
be the moment map corresponding to the boundary components of $\RS$, 
and set 
\begin{align*}
  \widetilde{\frakN}(\RS^+ \amalg \RS^-) 
  &= (\mu^G_{\partial \RS})^{-1}(\frakC) \\
  &= \bigcup_{t \in (0,1]} \biggl( \bigcup_{\alpha_{m+2}^+, \, \alpha_1^-} 
    \scN^G(\RS^+ ; t \bsalpha^+)
  \times \scN^G(\RS^- ; t \bsalpha^-) \biggr),
\end{align*}
where $\bsalpha^+ = (\alpha_1^+, \dots, \alpha^+_{m+2})$,
$\bsalpha^- = (\alpha_1^-, \dots, \alpha^-_{n-m})$
with $(\alpha_1^+, \dots, \alpha^+_{m+1}) = (\alpha_1, \dots, \alpha_{m+1})$, 
$(\alpha_2^-, \dots, \alpha^-_{n-m}) = (\alpha_{m+2}, \dots, \alpha_{n})$.
This space has an action of 
$G^{n+2} = \prod_{i=1}^{m+2} G^+_i \times \prod_{i=1}^{n-m} G^-_i$
and a $G^{n+2}$-invariant stratification induced from those on 
$\scN^G(\RS^{\pm}; \bsalpha^{\pm})$.
Note that the lower dimensional strata of 
$\widetilde{\frakN}(\RS^+ \amalg \RS^-)$  has the form
\[
  \bigcup_{t \in (0,1]}
  \scN^G(\RS^+ ; t\bsalpha^+)
  \times \Sing 
  \scN^G(\RS^- ; t\bsalpha^-)
\]
with $\alpha_1^- = \sum_{i = 2}^{n-m} \epsilon_i \alpha_i^-$ , or
\[
  \bigcup_{t \in (0,1]}
  \Sing
  \scN^G(\RS^+ ; t\bsalpha^+)
  \times
  \scN^G(\RS^- ; t\bsalpha^-)
\]
with $\alpha_{m+2}^+ = \sum_{i=1}^{m+1} \epsilon_i \alpha_i^+$
for some $\epsilon_i \in \{\pm 1 \}$.
From Proposition \ref{pr:crit_mu} and $|\bsalpha^{\pm}|<1$, 
trivialization \eqref{eq:triv_sing} lifts to that on 
$\bigcup_{t \in (0,1]} \Sing \scN^G(\RS^{\pm} ; t\bsalpha^{\pm})$
given by 
\[
  \Sing \scN^G(\RS^{\pm} ; \bsalpha^{\pm}) \longrightarrow
  \Sing \scN^G(\RS^{\pm} ; t\bsalpha^{\pm}),
  \quad
  (\bsa, \bsb) \longmapsto (\bsa, \bsb^t).
\]
From Corollary \ref{cor:stratification}, the space
$\widetilde{\frakN}(\RS^+ \amalg \RS^-)$ is locally homeomorphic to 
$V \times C(L)$ for some open set $V$ in a strata and 
a cone $C(L) = ([0,\infty) \times L)/(\{0 \} \times L)$ 
over a submanifold $L$.
We fix a $G^{n+2}$-invariant Riemannian metric 
on $\widetilde{\frakN}(\RS^+ \amalg \RS^-)$
such that it has the form  
$g_V + dr^2 + r^2 g_L$
on each neighborhood $V \times C(L)$ of the singular locus,
where $g_V$ and $g_L$ are $G^{n+2}$-invariant Riemannian metrics 
on $V$ and $L$, respectively,
and $r \in [0, \infty)$. 

Let $\nu_C^G : \scN^G(\RS^+ \amalg \RS^-) \to G_C$ be the moment map 
of the action of the diagonal subgroup $G_C \subset G_{m+2}^+ \times G_1^-$,
and define
\[
  \frakN^G(\RS^+ \amalg \RS^-) 
  = (\mu^G_{\partial \RS}, \nu_C^G)^{-1}(\frakC \times \{1\}) 
  = (\nu^G_C)^{-1}(1) \cap \widetilde{\frakN}(\RS^+ \amalg \RS^-)
\]
so that the family $\frakN(\RS) \to (0,1]$ is given by
\[
  \pi_{\frakC} \circ \mu^G_{\partial \RS} : 
  \frakN(\RS) \cong \frakN^G(\RS^+ \amalg \RS^-) / (G^{n} \times G_C)
  \longrightarrow \frakC
  \longrightarrow (0,1].
\]
Then the horizontal lift of the trivialization \eqref{eq:triv_sing}
of $\frakC \to (0,1]$ gives a $G^{n+1}$-equivariant trivialization
\begin{align}
  \psi_t: 
  \frakN^G(\RS^+ \amalg \RS^-)_1 
  & \longrightarrow 
  \frakN^G(\RS^+ \amalg \RS^-)_t, \notag \\ 
  \bigl( (\bsa^+, e^{\bsx^+}), (\bsa^-, e^{\bsx^-}) \bigr)
  & \longmapsto
  \bigl( (\bsc^+(\bsa, \bsx,t), e^{t \bsx^+}), 
         (\bsc^-(\bsa, \bsx,t), e^{t \bsx^-}) \bigr) \label{eq:triv_N^G}
\end{align}
of the family $\widetilde{\frakN}(\RS^+ \amalg \RS^-) \to (0,1]$
preserving the stratification,
where 
$\frakN^G(\RS^+ \amalg \RS^-)_t = 
 (\mu^G_{\partial \RS}, \nu^G_C)^{-1}(\scC_{t \bsalpha} \times \{ 1 \})$ 
is the fiber over $t \in (0,1]$.
Since $\psi_t$ is $G^{n+1}$-equivariant, it descends to a diffeomorphism
$\psi_t : \scN_{\bsalpha}(\RS) \to \scN_{t\bsalpha}(\RS)$.
From the construction of $\psi_t$, we have
\[
  \frac 1t \psi_t^* \widetilde{\vartheta}_{t, \bsalpha C} 
  = \frac 1t \psi_t^* \cos^{-1} \Bigl( \frac 12 \tr e^{x^+_{m+2}} \Bigr)
  = \frac 1t \cos^{-1} \Bigl( \frac 12 \tr e^{t x^+_{m+2}} \Bigr)
  = \widetilde{\vartheta}_{\bsalpha, C},
\]
which completes the proof.
\end{proof}

\begin{remark}
From \eqref{eq:triv_N^G}, the flow $\psi_t$ preserves the subfamily
\begin{align*}
  \frakN^T(\RS^+ \amalg \RS^-) 
  &= \bigcup_{t \in (0,1]} 
    (\mu^T_{\partial \RS}, \nu_T^G)^{-1}
    (e^{-t \alpha_1 x_0}, \dots, e^{-t \alpha_1 x_0}, 1) \\ \notag
  &= \frakN^G(\RS^+ \amalg \RS^-) \cap \scN^T(\RS^+ \amalg \RS^-)
\end{align*}
of $\scN^G(\RS^+ \amalg \RS^-)$.
The flow $\psi_t$ restricted to $\frakN^T(\RS^+ \amalg \RS^-)$ 
is also equivariant under the action of 
$T_{m+2}^+ \times \{1\} \subset G_{m+2}^+ \times G_1^-$,
and hence $\psi_t : \scN_{\bsalpha}(\RS) \to \scN_{t \bsalpha}(\RS)$ is
equivariant under the action of the Hamiltonian $S^1$-action of $\vartheta_C$.
\end{remark}

\begin{proof}[Proof of Theorem \ref{thm:isom1}]
Let $\RS = \bigcup_{i=1}^{n-2} \RS_i$ be the pair-of-pants decomposition
given by $\Gamma$.
For the group valued moment map
\[
  \mu^G = (\mu^G_{\partial \RS} , \nu_{C_1}^G, \dots, \nu_{C_{n-3}}^G)
  : \scN^G(\RS_1 \amalg \dots \amalg \RS_{n-2})
  \longrightarrow G^n \times G^{n-3}
\]
of the $G^n \times G^{n-3}$-action, we define
$\frakN^G(\coprod_i \RS_i) = (\mu^G)^{-1}(\frakC \times \{ \bsone \})$
and 
$\frakN^T(\coprod_i \RS_i) = \frakN^G(\coprod_i \RS_i) \cap \scN^T(\coprod_i \RS_i)$.
By applying the above argument, we obtain a trivialization 
$\psi_t : \frakN^G(\coprod_i \RS_i)_1 \to \frakN^G(\coprod_i \RS_i)_t$
of $\frakN^G(\coprod_i \RS_i)$
which induces trivializations of 
$\frakN^T(\coprod_i \RS_i)$
and $\frakN(\RS)$, and satisfies
\[
  \frac 1t \psi_t^* \Theta_{t\bsalpha, \Gamma} = 
  \Theta_{\bsalpha, \Gamma}.
\]
In particular, $\psi_i$ preserves the action variables on $\scN_{\bsalpha}(\RS)$.

The Hamiltonian torus action of the Goldman system, which is 
defined on an open dense subset $U \subset \scN_{\bsalpha}(\RS)$, 
is induced from the action of a maximal torus 
$\prod_{i=1}^{n-3} (T_{C_i^+} \times \{1\})$ in 
$\prod_i (G_{C_i^+} \times \{1\}) \subset \prod_i (G_{C_i^+} \times G_{C_i^-})$.
Since the trivialization
$\psi_t : \frakN^T(\coprod_i \RS_i)_1 \to \frakN^T(\coprod_i \RS_i)_t$
is $\prod_i (T_{C_i^+} \times T_{C_i^-})$-equivariant, 
the Hamiltonian torus action of the Goldman systems are preserved by $\psi_t$.
This means that $\psi_t : \scN_{\bsalpha}(\RS) \to \scN_{t \bsalpha}(\RS)$ 
 preserves angle variables.
Hence $(1/t) \psi_t^* \omega_{\scN_{t \bsalpha}}$ coincides with
$\omega_{\scN_{\bsalpha}}$ on $U$.
Since $\psi_t$ is a diffeomorphism 
and $U$ is dense,
we have $(1/t) \psi_t^* \omega_{\scN_{t \bsalpha}} 
= \omega_{\scN_{\bsalpha}}$
on $\scN_{\bsalpha}(\RS)$.
\end{proof}

\section{Goldman systems and bending systems} \label{sc:Goldman-bending}

We see in \pref{th:comparison} that 
$\scN_{\bsalpha}$ is isomorphic to $\scM_{\bsalpha}$ as complex manifolds
if $|\bsalpha|<1$.
On the other hand, 
Jeffrey \cite{Jeffrey_EMS} proved the following
by using the $\frakg$-extended moduli space.

\begin{proposition}[{Jeffrey \cite[Theorem 6.6]{Jeffrey_EMS}}]
\label{prop:Jeffrey_EMS}
For sufficiently small $\bsalpha \in (0,1/2)^n$, 
the moduli space $\scN_{\bsalpha}(\RS)$ is symplectomorphic to $\scM_{\bsalpha}$.
\end{proposition}

\begin{proof}[Outline of the proof]
The proposition is proved
by using a canonical local model of Hamiltonian spaces
called the Marle-Guillemin-Sternberg form
\cite{Guillemin-Sternberg_NFMM, Marle}.
Recall that the moment map of the $G^n$-action on the $\frakg$-extended moduli space 
$\scN^{\frakg} (\RS)$ is given by 
\[
  \mu^{\frakg}: 
  \scN^{\frakg}(\RS) \longrightarrow \frakg^n, \quad
  (\bsa, \bsx) \longmapsto - \bsx.
\]
Since the stabilizer of $(\bsone, \bszero) \in  (\mu^{\frakg})^{-1}(\bszero)$ is 
the diagonal subgroup $G \subset G^n$, the fiber $(\mu^{\frakg})^{-1}(\bszero)$ 
is identified with $G^n/G$ by
\[
  G^n/G \longrightarrow (\mu^{\frakg})^{-1}(\bszero), \quad
  [\sigma_1, \sigma_2, \dots, \sigma_n] \longrightarrow
  (\sigma_1 \sigma_2^{-1}, \dots, \sigma_1 \sigma_n^{-1}).
\]
Then the Marle-Guillemin-Sternberg form of 
a neighborhood of $(\mu^{\frakg})^{-1}(\bszero)$ is 
a neighborhood of the zero section of the vector bundle 
$G^n \times_G (\frakg^n/\frakg)^* \to G^n/G$
equipped with the moment map
\[
  \mu_{\mathrm{MGS}}: G^n \times_G (\frakg^n/\frakg)^*  \longrightarrow
  \frakg^n, \quad
  [\bssigma, \bsy] \longrightarrow (\Ad(\sigma_i)y_i)_i.
\]
This implies that
\begin{align*}
  (\mu^{\frakg})^{-1}(\scO_{\bsalpha})/G^n
  &= (\mu_{\mathrm{MGS}})^{-1}(\scO_{\bsalpha})/G^n \\
  &= \{ [\bssigma, \bsy] \in G^n \times_G (\frakg^n/\frakg)^*\, | \,
        \Ad(\sigma_i)y_i \in \scO_{\alpha_i}, i=1, \dots, n \} /G^n \\
  &\cong (\scO_{\bsalpha} \cap \{(x, \dots, x) \in \frakg^n 
        \,| \, x \in \frakg\}^{\perp})/G \\
  &= \{ (x_1, \dots, x_n) \in \scO_{\bsalpha}\, | \,  
        x_1 + \dots + x_n = 0 \}/G \\
  &= \scM_{\bsalpha}.
\end{align*}
\end{proof}

Fix $\bsalpha$ such that $|\bsalpha|<1$,
and consider the family
\[
  f: \frakN (\RS) = \bigcup_{t \in (0,1]} 
  (\scN_{t \bsalpha} (\RS), (1/t)\omega_{\scN_{t \bsalpha}})
  \longrightarrow (0,1]
\]
of symplectic manifolds.
From Proposition \ref{prop:Jeffrey_EMS}, 
a fiber $(\scN_{t\bsalpha}, \omega_{\scN_{t\bsalpha}})$ 
over sufficiently small $t \in (0,1]$
is symplectomorphic to
$(\scM_{t\bsalpha}, \omega_{\scM_{t\bsalpha}})$.
Since $(\scM_{\bsalpha}, \omega_{\scM_{\bsalpha}})$ is 
symplectomorphic to 
$(\scM_{t \bsalpha}, (1/t) \omega_{\scM_{t \bsalpha}})$
by scaling $\bsx \longmapsto t \bsx$,
we can extend the family $f : \frakN(\RS) \to (0,1]$
to a family over $[0,1]$ by setting 
\[
  f^{-1}(0) = (\scM_{\bsalpha}, \omega_{\scM_{\bsalpha}}).
\]

\begin{proposition}
The symplectic trivialization $\{\psi_t\}$ of $\frakN(\RS) \to (0,1]$ 
given in Theorem \ref{thm:isom1}
extends to the family over $[0,1]$.
Moreover this trivialization identifies
Goldman systems 
$(1/t)\Theta_{t\bsalpha, \Gamma} : \scN_{t\bsalpha} \to \bR^{n-3}$ 
and the bending system
$\Phi_{\Gamma}: \scM_{\bsalpha} \to \bR^{n-3}$.
\end{proposition}

\begin{proof}
Fix $\bsg \in \scN_{\bsalpha}$ and let
\begin{equation}
  \bsg(t) = (g_1(t), \dots, g_n(t)) =
  (e^{x_1(t)}, \dots, e^{x_n(t)}) 
  := \psi_t (\bsg) \in \scN_{t \bsalpha}
  \label{eq:trivialization_g}
\end{equation}
be the trajectory of $\psi_t$
starting from $\bsg$.
Then $\bsx(t) = (x_1(t), \dots, x_n(t))$ is a smooth curve in 
$\bigcup_t \scO_{t \bsalpha}$ of the form
$x_i(t) = t x_i + O(t^2)$.
Since 
$g_1(t) \dots g_n(t) = 1 + t(x_1 + \dots + x_n) + O(t^2)$, 
the point $\bsx = (x_1, \dots, x_n)$
lies in $\scM_{\bsalpha}$.
We also take smooth families of tangent vectors 
$u(t), v(t) \in T_{\bsg(t)} \scN_{t\bsalpha}$ such that
$d\psi_t (u(1)) = u(t)$ and $d\psi_t (v(1)) = v(t)$.
Then
\[
  \frac 1t \omega_{\scN_{t\bsalpha}}(u(t), v(t))
  = \omega_{\scN_{\bsalpha}}(u(1), v(1))
\]
for all $t \in (0,1]$.
Let $\bsxi(t) = \bsxi + O(t)$, $\bseta(t) = \bseta + O(t)$
be smooth curves in $\frakg^n$ such that
\begin{align*}
  u(t)(\gamma_i) &= \Ad_{g_i(t)}\xi_i(t) - \xi_i(t),\\
  v(t)(\gamma_i) &= \Ad_{g_i(t)}\eta_i(t) - \eta_i(t).
\end{align*}
Since
\begin{equation}
  \Ad_{g_i(t)} \xi_i(t) - \xi_i(t)
  = t[x_i, \xi_i] + O(t^2),
  \label{eq:asympt_u}
\end{equation}
$\bsxi = (\xi_1, \dots, \xi_n)$, 
$\bseta = (\eta_1, \dots, \eta_n) \in \frakg^n$ give 
tangent vectors of $\scM_{\bsalpha}$ at $\bsx$.
Note that \eqref{eq:asympt_u} also implies  that
$$
  (u(t) \cup v(t) ) [\pi_1(\RS), \partial \pi_1(\RS)] )
  = O(t^2).
$$
On the other hand, the second term of $\omega_{\scN_{t \bsalpha}}$
in (\ref{eq:symp_par}) has the form
$$
  \frac 12 \sum_{i=1}^n
      (\langle \xi_i(t), \Ad_{g_i(t)}\eta_i(t) \rangle
      - \langle \eta_i(t), \Ad_{g_i(t)}\xi_i(t) \rangle )
  = t \sum_{i=1}^n \langle x_i, [\xi_i, \eta_i] \rangle
  + O(t^2).
$$
Thus we have
$$
  \frac 1t \omega_{\scN_{t\bsalpha}}(u(t), v(t)) 
  = \omega_{\scM_{\bsalpha}}(\bsxi, \bseta) + O(t).
$$
Since the left hand side is independent of $t$,
we have
$$
  \frac 1t \omega_{\scN_{t\bsalpha}}(u(t), v(t)) 
  = \omega_{\scM_{\bsalpha}}(\bsxi, \bseta),
$$
or equivalently 
$\psi_0^* \omega_{\scM_{\bsalpha}} = \omega_{\scN_{\bsalpha}}$.

Next we show that the integrable systems are identified.
Suppose that the $k$-th boundary component $C_k$ is given by
$[C_k] = \gamma_{i_k} \dots \gamma_{i_k+n_k}$.
If we write
$$
  g_{i_k}(t) \dots g_{i_k+n_k}(t) = e^{y_k(t)}
$$
for $y_k(t) \in \frakg$, then $y_k(t)$ has eigenvalues
$\pm \vartheta_{t\bsalpha , C_k}(\bsg(t))$.
Since 
$$
  y_k(t) = t(x_{i_k} + \dots + x_{i_k+n_k}) + O(t^2)
$$ 
and the eigenvalues of $x_{i_k} + \dots + x_{i_k+n_k}$ are
$\pm \varphi_{d_k} (\bsx)$, 
we have
$$
  \frac 1t \vartheta_{t \bsalpha, C_k}(\bsg(t)) 
  = \varphi_{d_k}(\bsx) + O(t).
$$
Theorem \ref{thm:isom1} implies that
the left hand side is also independent of $t$,
and hence $\frac{1}{t} \vartheta_{t \bsalpha, C_k}$ is 
identified with $\varphi_{d_k}$
by the trivialization.
\end{proof}



\bibliographystyle{amsalpha}
\bibliography{bibs}

\noindent
Yuichi Nohara

Faculty of Education,
Kagawa University,
Saiwai-cho 1-1,
Takamatsu,
Kagawa,
760-8522,
Japan.

{\em e-mail address}\ : \  nohara@ed.kagawa-u.ac.jp
\ \vspace{0mm} \\

\noindent
Kazushi Ueda

Department of Mathematics,
Graduate School of Science,
Osaka University,
Machikaneyama 1-1,
Toyonaka,
Osaka,
560-0043,
Japan.

{\em e-mail address}\ : \  kazushi@math.sci.osaka-u.ac.jp
\ \vspace{0mm} \\

\end{document}